\documentclass[numbers,webpdf,imaiai]{ima-authoring-template}
\usepackage{amsmath,amsfonts,amsthm}

\makeatletter
\renewcommand*\l@subsubsection[2]{{}{}{}}
\makeatother

\makeatletter
\AtBeginDocument{%
  \renewenvironment{proof}[1][\proofname]{%
    \par\removelastskip
    \pushQED{\qed}%
    \normalfont
    \topsep 7.5\p@ \@plus 7.5\p@\relax
    \trivlist
    \item[\hskip\labelsep\itshape
          #1\@addpunct{.}\ ]%
    \ignorespaces
  }{%
    \popQED
    \endtrivlist
    \@endpefalse
  }%
}
\makeatother

\usepackage{etoolbox}

\makeatletter
\patchcmd{\@maketitle}
  {[\@history]}
  {\ifx\@history\@empty\else[\@history]\fi}
  {}{\PackageWarning{V14_IMA}{Could not patch the manuscript-history field}}
\makeatother

\usepackage{subcaption}
\usepackage{textcomp}
\usepackage{stfloats}
\usepackage{url}
\usepackage{verbatim}
\usepackage{graphicx}
\usepackage{tikz}
\usetikzlibrary{arrows.meta,calc,positioning}

\DeclareMathOperator*{\argmin}{arg\,min}

\providecommand{\norm}[1]{%
  \ensuremath{\left\lVert #1 \right\rVert}%
}

\providecommand{\abs}[1]{%
  \ensuremath{\left\lvert #1 \right\rvert}%
}

\let\originalleft\left
\let\originalright\right
\renewcommand{\left}{\mathopen{}\mathclose\bgroup\originalleft}
\renewcommand{\right}{\aftergroup\egroup\originalright}

\usepackage{mathtools}

\theoremstyle{thmstyletwo}
\newtheorem{theorem}{Theorem}
\newtheorem{lemma}{Lemma}
\newtheorem{proposition}{Proposition}
\newtheorem{corollary}{Corollary}
\theoremstyle{thmstylethree}
\newtheorem{assumption}{Assumption}

\newtheorem{remark}{Remark}

\numberwithin{equation}{section}

\begin{document}

\copyrightyear{2026}

\title[Lasso Universality Under Linearly Dependent Covariates]
{Lasso Universality Under Linearly Dependent Covariates in the Sparse Regime}

\authormark{Mesforush and Parhi}

\author{Soroush Mesforush\ORCID{0009-0006-5821-2294}* and Rahul Parhi\ORCID{0000-0002-1971-7699}
\address{%
\orgname{University of California, San Diego},
\orgaddress{%
\state{La Jolla, CA} \postcode{92093},
\country{USA}}}}

\corresp[*]{Corresponding author:
\url{smesforush@ucsd.edu}}

\abstract{
Throughout the last decade, Gaussian universality has been widely studied for high-dimensional estimation problems. Most of the literature focuses on i.i.d.\
sensing matrices or accounts for special forms of dependence, such as block dependence or other specific row/column dependencies. More general simultaneous row and column mixing has not yet been fully studied.
In this paper, we focus on that setting.
We prove a Gaussian universality theorem for the lasso in the sparse regime, where the non-Gaussian covariates have linearly dependent rows and columns. To the best of our knowledge, our setting permits a broader simultaneous row and column dependence structure than those treated in much of the prior universality literature.
%
Numerical illustrations for various sparse profiles support the universality claims of this paper.
}


\maketitle

\section{Introduction}\label{Sec:SecI}
High-dimensional statistics and compressed sensing often consider problems in which the ambient dimension is comparable to or larger than the sample size~\cite{maleki2026high}. These problems are typically formulated as estimating $\beta^\star \in \mathbb{R}^d$ from the noisy observations
\begin{equation}\label{eq:mainproblem}
    y = X\beta^\star+\xi \in \mathbb{R}^n,
\end{equation}
where $X\in\mathbb{R}^{n\times d}$ and $\xi\in\mathbb{R}^n$ is additive noise, typically modeled as $\xi\sim\mathcal{N}(0,\sigma^2 I_n)$, $\sigma > 0$, independent of $X$. The proportional high-dimensional regime considered here is specified precisely in Section~\ref{Sec:SectionII}. For example, in linear regression problems, $X$ is the design matrix and $\beta^\star$ is the coefficient vector. Meanwhile, in signal recovery problems, $X$ is the sensing matrix and $\beta^\star$ is the signal of interest.

This model is applicable in a broad range of high-dimensional signal recovery problems, see, e.g.,~\cite{tibshirani1996regression,bayati2011lasso,candes2007dantzig,dobson2018introduction,donoho2009message,thrampoulidis2018precise,hu2019asymptotics,donoho2011noise,tsuda2026universality}. 
The analyses of these papers investigate properties of estimators such as the lasso. These have led to precise characterizations of estimation error, risk formulas, noise sensitivity, and phase transition behavior. 



The problem of estimating $\beta^\star$ from the observations \eqref{eq:mainproblem} is intractable without structural assumptions on the design/sensing matrix.
%
Strong assumptions on $X$ have therefore been widely adopted in the literature. For example, it is fairly common to assume that the entries of $X$ are i.i.d.\ Gaussian random variables. In reality, $X$ is almost always structured, with dependence between the entries~\cite{krahmer2014structured,do2011fast,dudeja2022universality,dudeja2023universality}. This has led to frameworks for studying such dependence, including approximate message passing~\cite{donoho2009message} and the convex Gaussian min--max theorem~\cite{celentano2023lasso,gerbelot2023graph}.
%

Another line of work considers distributional assumptions on the design matrix, where the entries may be non-Gaussian, non-i.i.d., and dependent. Recent work on \emph{universality phenomena} has demonstrated that, in appropriate asymptotic regimes, these design matrices behave the same as i.i.d.\ Gaussian designs~\cite{korada2011applications,PanahiHassibiLS,MontanariElastic,bayati2015universality,oymak2018universality,wang2024universality,han2023universality,han2025entrywise,han2025long,HuTIT,gerace2024gaussian,ghane2024universality,LahiryTIT,dudeja2022universality,dudeja2023universality,dudeja2024spectral,moniri_hassani_dependent_ridge,wen2025does,tsuda2026universality}. These are referred to as universality phenomena as they define universal laws that control the overall behavior of a system, agnostic to the microscopic subtleties.

There have been many attempts to study universality under different dependent designs. For example, block-dependent designs and other structured designs are investigated in~\cite{LahiryTIT,tsuda2026universality,dudeja2024spectral}. However, to the best of our knowledge, there is a gap in the literature when considering general row/column dependencies, which motivates our investigation, particularly in the \emph{sparse} regime. Indeed, in many settings following the observation model \eqref{eq:mainproblem}, it is natural to assume that $\beta^\star$ is sparse. Applications include compressed sensing \cite{DonohoCS,candesRombergTao,candes2007sparsity}, medical and computational imaging~\cite{lustig2008compressed,duarte2008single}, image reconstruction \cite{PritchardDenoiser}, wireless channel estimation~\cite{BajwaSayeedNowakCS}, and signal detection~\cite{MyPaper2024}. The typical way to solve \eqref{eq:mainproblem} in sparse settings is with $\ell^1$-regularization, often called the \emph{lasso}~\cite{tibshirani1996regression}, given by
\begin{equation}\label{eq:lassoproblem}
\hat{\beta} \in \argmin_{\beta\in\mathbb{R}^d}\left\{\frac 1{2n} \left\|y-X\beta\right\|_2^2+\lambda\|\beta\|_1\right\}, \quad \lambda > 0.
\end{equation}

Thus, we investigate the universality property of the squared error (SE) in estimating $\beta^\star$ via the lasso. We investigate general notions of row and column linear dependence of $X$. Following~\cite{moniri_hassani_dependent_ridge}, we consider the model $X = AZB$ in which $Z\in \mathbb{R}^{n\times d}$  is a random sub-Gaussian matrix with i.i.d.\ entries, and $A\in\mathbb{R}^{n\times n}$ and $B\in\mathbb{R}^{d\times d}$ are matrices controlling the mixing of the covariates.
%
In this paper, we focus on the following question: 
\begin{center}
    \itshape In the sparse regime, when are the lasso solutions and squared error universal under linearly dependent covariates?
\end{center}

We remark that a lot of effort has been put into support recovery via the lasso, where the goal is to recover the locations of the nonzero entries of $\beta^\star$. Support recovery generally requires conditions where there are not excessive correlations between ``active'' and ``inactive'' covariates~\cite{zhao2006model,wainwright2009sharp}. Under these conditions, the lasso solution is unique and exactly recovers the signed support. This is done via primal-dual witness (PDW) construction proposed in~\cite{wainwright2009sharp}, which involves solving an optimization problem restricted to the true support of $\beta^\star$ and then verifying that the inactive coordinates satisfy strict dual feasibility. 
Although the goal of this paper is not support recovery, it turns out that there is an interesting connection between the PDW framework and universality, which we crucially exploit.




%


\subsection{Contributions and Outline}


Motivated by the gap in universality theorems and inspired by the PDW framework, we consider a reduction-based argument similar to~\cite{wainwright2009sharp}, but in the context of universality, which we next outline. We begin by defining the ``restricted'' or ``oracle'' lasso estimator, which has knowledge of the support:
\begin{equation}\label{eq:OracleObj}
\hat{\beta}^{\mathrm{oracle}} \in \argmin_{\substack{\beta \in \mathbb{R}^d \\ \operatorname{supp}(\beta) \subseteq S}}\left\{\frac 1{2n} \left\|y-X\beta\right\|_2^2+\lambda\|\beta\|_1\right\},
\end{equation}
where $S \coloneqq \operatorname{supp}(\beta^\star)  \coloneqq \{i\in\{1,\ldots,d\} : \beta_i^\star\neq 0\}$. We first prove in Proposition~\ref{prop:RestrictedProp}, a universality result for \eqref{eq:OracleObj}, under the design $X = AZB$ discussed above. In particular, we show that \eqref{eq:OracleObj} is concentrated around the deterministic problem
\begin{equation}\label{eq:detlasso}
\hat{\beta}^\mathrm{det} \in \argmin_{\beta\in \mathbb{R}^k} \left\{\frac 1{2n} \mathrm{tr}(A^\top A) \left\|B_S(\beta_S^\star-\beta)\right\|_2^2+\lambda\|\beta\|_1\right\},
\end{equation}
where $k = |S|$, the vector $\beta^\star_S \in \mathbb{R}^k$ denotes the restriction of $\beta^\star$ to the support set $S$, and the matrix $B_S \in \mathbb{R}^{d \times k}$ denotes the corresponding active submatrix (see Section~\ref{Sec:Notations} for a precise definition). 

Next, we consider the setting inspired by the PDW framework, which holds for many reasonable choices of $A$ and $B$. We prove in Theorem~\ref{thm:lassoFullUniversality} that, under these conditions, the full lasso \eqref{eq:lassoproblem}, and the oracle lasso \eqref{eq:OracleObj} solutions are equal with high probability and, in particular, the previous universality result can be readily transferred to the full lasso solution \eqref{eq:lassoproblem}. 

We reiterate that the design dependency considered in this paper allows general row and column dependence. This takes a step beyond i.i.d., block-dependent, and structured designs considered in previous works. Finally, through simulations, we provide illustrations of the lasso SE for both the oracle \eqref{eq:OracleObj}, and the full lasso estimators \eqref{eq:lassoproblem} under various choices of designs compatible with our framework.


\subsection{Related work}

\paragraph{Classical Universality.} A classical route to universality is using replacement and swapping techniques such as Lindeberg's method \cite{lindeberg1922neue,chatterjee2006lindeberg}. The seminal work of \cite{Donoho2009} on phase transitions in sparse recovery paved the way to study universality in various contexts. In particular, ``classical'' universality results crucially rely on design matrices having i.i.d.\ entries. For example, universality in convex regression problems was discussed in \cite{PanahiHassibiLS}. Meanwhile, the universality of the elastic-net error is studied in \cite{MontanariElastic}. Furthermore, universality has been investigated in concepts such as polytope geometry~\cite{bayati2015universality}, and a general universality framework for regularized estimators is derived in~\cite{han2023universality}.


\paragraph{Dependent Designs.}
As previously discussed, there has been preliminary work on universality under dependent designs~\cite{dudeja2024spectral,tsuda2026universality,dudeja2022universality,moniri_hassani_dependent_ridge,HuTIT,LahiryTIT,dudeja2023universality,wang2024universality}. For example,~\cite{LahiryTIT} extends Gaussian universality for linear regression to block-dependent linear models. A follow-up work \cite{tsuda2026universality}, inspired by \cite{LahiryTIT}, considers the block-dependence idea and formulates a broader universality framework for linear estimators, directly extending \cite{han2023universality}. Furthermore, universality in linearized message passing for phase retrieval with structured design matrices is discussed in \cite{dudeja2022universality,dudeja2023universality}. Moreover, \cite{dudeja2024spectral} introduces a notion of universality classes for design matrices and argues that matrices that fall in their class yield similar regularized least squares performance in the high-dimensional regime. These classes
are discussed broadly, where the left linear transform of i.i.d.\ matrices, along with the model in \cite{moniri_hassani_dependent_ridge}, motivated the model used in this paper. Lastly, \cite{moniri_hassani_dependent_ridge} considers linearly dependent covariates, and proves the universality of ridge regression. In a nutshell, these papers provide initial results for universality under dependent designs.

\paragraph{Other Notions of Universality.}
Universality has applications that extend beyond those discussed so far. Universality laws governing random feature, perceptron, and transfer learning models have been investigated. For example, \cite{HuTIT} shows a random feature model is equivalent to a linear Gaussian model in the context of generalization and training errors. This is true when the covariance matrices match. Their proof is based on the Lindeberg principle discussed beforehand. Furthermore,  Gaussian universality for perceptrons and transfer learning is studied in \cite{gerace2024gaussian,ghane2024universality}. 




\subsection{Notation}\label{Sec:Notations}

The following notation is used throughout this paper:  For an index set $T\subset[n]$, we denote the subvector of $v$ indexed by $T$ as $v_T$. Similarly, for any matrix $X$, we denote the submatrix formed by the columns indexed by $T$ as $X_{T}$.  Moreover, for a square matrix $Q\in\mathbb{R}^{n\times n}$ we denote the spectral and Frobenius norm as  $\|Q\|_\mathrm{op}$ and $\|Q\|_F$, respectively. Moreover, when $Q$ is symmetric, we define its smallest eigenvalue as $\lambda_\mathrm
{min}(Q)$. The $K \times K$ identity matrix is represented by $I_K$. The Kronecker product of two matrices $X$ and $Y$ and the vectorization operator that transforms the matrix $X$ to a single column vector are denoted by $X \otimes Y$ and $\mathrm{vec}(X)$, respectively. On an event $A$, the indicator function is defined as $\mathbf{1}_A$. Given sequences $f(n)$ and $g(n)$, the notation $f(n) = o\left(g(n)\right)$ means $\lim_{n\to\infty} f(n)/g(n) = 0$ and the notation $f(n) = O(g(n))$ means that $\sup_n\left|\frac{f(n)}{g(n)}\right|<\infty$. Furthermore for a sequence of random variables $X_n$ the notation $X_n=O_\mathbb{P}(a_n)$ and $X_n=o_\mathbb{P}(a_n)$ as $n\to \infty$ denote stochastic boundedness and convergence in probability. We also often use the notation such as $y(Z) = X(Z)\beta^\star+\xi$ to explicate the dependence on the randomness of the problem. When it is clear from context, we drop the dependence for brevity.

\section{Main Universality Results}\label{Sec:SectionII}

In this section, we introduce the main results and an overview of the route taken to obtain these results. As previously discussed, the motivation that inspired this proof structure is based on the PDW framework \cite{wainwright2009sharp}. We show that the lasso solution and its SE remain agnostic to the randomness of the linearly dependent design matrix in the sparse regime. In particular, this paper exclusively focuses on the setting
\begin{equation}
    \frac d n \to \gamma\in(0,\infty) \quad\text{as}\quad n,d\to \infty.
\end{equation}

The way we obtain this result is a two-step procedure. First, we prove a universality result for the oracle solution \eqref{eq:OracleObj}, then, using a weaker version of the mutual incoherence condition from the PDW framework~\cite{wainwright2009sharp}, we prove that the restricted universality result transfers to the full lasso solution \eqref{eq:lassoproblem}. A high-level diagram illustrating the proof architecture is in Figure~\ref{fig:diagramproc}. 
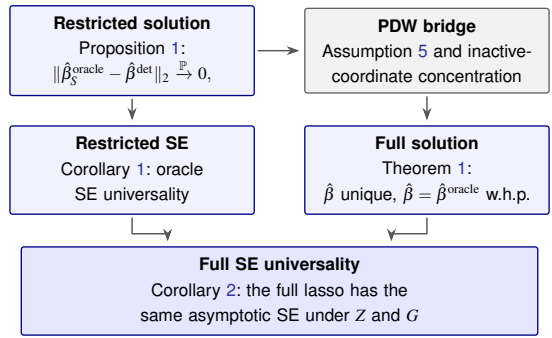
\begin{figure}[htb!]
\centering
\begin{tikzpicture}[
    x=1cm,
    y=1cm,
    >={Stealth[length=1.7mm,width=1.2mm]},
    roadmap/.style={
        draw=black!62,
        line width=0.5pt,
        rounded corners=1pt,
        fill=black!1,
        align=center,
        inner xsep=4pt,
        inner ysep=4pt,
        font=\sffamily\scriptsize,
        text width=2.95cm,
        minimum height=1.04cm
    },
    bridge/.style={roadmap,fill=black!5},
    outcome/.style={roadmap,draw=blue!52!black,fill=blue!5},
    finaloutcome/.style={outcome,fill=blue!10,text width=6.55cm,minimum height=1.02cm},
    flow/.style={
        ->,
        draw=black!66,
        line width=0.55pt,
        line cap=round,
        line join=round,
        shorten <=1.2pt,
        shorten >=1.8pt
    }
]
\node[outcome] (prop) at (-1.95,0)
    {\textbf{Restricted solution}\\
     Proposition~\ref{prop:RestrictedProp}: $\|\hat{\beta}^\mathrm{oracle}_S -{\hat{\beta}}^\mathrm{det}\|_2\xrightarrow{\mathbb{P}} 0,$};
\node[outcome] (corone) at (-1.95,-1.58)
    {\textbf{Restricted SE}\\
     Corollary~\ref{cor:Corollary}: oracle SE universality};
\node[bridge] (pdw) at (1.95,0)
    {\textbf{PDW bridge}\\
     Assumption~\ref{ass:weirdstrong} and inactive-coordinate concentration};
\node[outcome] (theorem) at (1.95,-1.58)
    {\textbf{Full solution}\\
     Theorem~\ref{thm:lassoFullUniversality}: \\ $\hat\beta$ unique, $\hat{\beta}=\hat\beta^{\mathrm{oracle}}$ w.h.p.};
\node[finaloutcome] (cortwo) at (0,-3.18)
    {\textbf{Full SE universality}\\
     Corollary~\ref{cor:Corollary2}: the full lasso has the same asymptotic SE under $Z$ and $G$};

\draw[flow] (prop.east) -- (pdw.west);
\draw[flow] (prop.south) -- (corone.north);
\draw[flow] (pdw.south) -- (theorem.north);
\draw[flow] (corone.south) -- ++(0,-0.25)
    -| ($(cortwo.north)+(-1.48,0)$);
\draw[flow] (theorem.south) -- ++(0,-0.25)
    -| ($(cortwo.north)+(1.48,0)$);
\end{tikzpicture}
\caption{Proof roadmap.}
\label{fig:diagramproc}
\end{figure}

Before formally stating the main results, we consider the main definitions and assumptions governing this work. 
\begin{assumption}\label{ass:Sparse}
The ground truth signal vector $\beta^\star$ is $k$-sparse, with support $S=\operatorname{supp}(\beta^\star)$, and satisfies $\|\beta^\star\|_2^2=\alpha^2$, $k=o(n)$. 

\end{assumption}
\begin{remark}\label{rem:remarkaboutsignals}
Assumption~\ref{ass:Sparse} holds for a wide range of sparse signals; see, e.g.,~\cite{dai2009subspace}. For example, the following signal profiles satisfy Assumption~\ref{ass:Sparse}, where $S=\{s_1,\dots,s_k\}$  denotes the support set.

\begin{itemize}
    \item \textbf{Geometrically Decaying Profile.}
    For signs $\sigma_i\in\{-1,1\}$ and any fixed $r\in(0,1)$,
\begin{equation}\label{eq:geomsig}
\beta^\star_{s_i} = \sigma_i\alpha \sqrt{\frac{1-r^2}{1-r^{2k}}}r^{i-1},\quad i=1,\dots,k.
\end{equation}
\item 
\textbf{Anchor-Diffuse Profile.}
An anchor coordinate is considered, and a fixed fraction of the signal energy is dedicated to it. The remaining energy is distributed along the other coordinates. This model can be defined without loss of generality for $k\geq2$, $\vartheta\in[0.5,1)$, $r\in (0,1)$, and $\sigma_{i}\in\{-1,1\}$:
\begin{equation}\label{eq:anchorsig}
\beta^\star_{s_i}=\begin{cases}
\sigma_{1}\alpha\sqrt{\vartheta},\quad i= 1\\ 
  \sigma_{i}\alpha\sqrt{\frac{(1-\vartheta)(1-r^2)}{1-r^{2(k-1)}}}r^{i-2},\quad i=2,...,k.
\end{cases} 
\end{equation}
\item 
\textbf{Normalized Power Profile.}
Let $p>\frac 1 2$, and $\sigma_{i}\in\{-1,1\}$. Then,
\begin{equation}\label{eq:powersig}
\beta^\star_{s_i}=\sigma_i \alpha \frac{i^{-p}}{\left(\sum_{\ell=1}^k \ell^{-2p}\right)^\frac 1 2},\quad i=1,\dots,k.
\end{equation}
\end{itemize}


As we shall see later, all of these signals (as well as other kinds sparse signals) fit the framework discussed in this paper.
\end{remark}

\begin{assumption}\label{ass:covstruct}
Let $X = AZB$, where $A\in\mathbb{R}^{n\times n}$ and $B\in\mathbb{R}^{d\times d}$ are deterministic matrices, each uniformly bounded in operator norm, that control the row/column dependence of the covariates. Furthermore, $Z$ is a mean-zero, unit-variance sub-Gaussian random matrix with i.i.d.\ entries that satisfies $\sup_{n,i,j} \|Z_{ij}\|_{\psi_2}<\infty$, where we recall the sub-Gaussian norm of a random variable $X$ is given by $\|X\|_{\psi_2} \coloneqq \inf\{c>0 \: : \: \mathbb{E}[\exp(X^2 / c^2)]\leq 2\}$. We refer to $Z$ as the \emph{latent variable matrix} from here onward.
\end{assumption}

\begin{assumption}\label{def:empspectral}
Let the eigenvalues of $A^\top A$ be $a_1,\dots a_n$ and those of $B^\top B$ be $b_1,\dots,b_d$. Suppose there exist probability measures $\mu_A$ and $\mu_B$ such that
\begin{equation}
\frac 1 n \sum_{i=1}^n\delta_{a_i}\implies \mu_A,\quad \frac 1 d \sum_{i=1}^d\delta_{b_i}\implies \mu_B,
\end{equation}
where $\implies$ denotes weak convergence of measures. These are referred to as the \emph{spectral distributions}.
\end{assumption}
Assumption~\ref{def:empspectral} is used only for the examples involving limiting spectra; it is not needed for the main universality results.
Next, we discuss the two main assumptions considered in the rest of the paper. 
\begin{assumption}\label{ass:lassoUniqueness}
There exists a constant $\kappa>0$, such that
\begin{equation}
\lambda_\mathrm{min}\left(\frac 1 n\mathrm{tr}(A^\top A) B_S^\top B_S\right)\geq \kappa.
\end{equation}
\end{assumption}

\begin{assumption}\label{ass:weirdstrong}
Let $\lambda>0$ be fixed. There exist constants $\eta\in(0,1)$ and $n_0$ such that, for all $n\geq n_0$,
\begin{equation}
\max_{j\in S^c}\left|\frac{1}{n}\mathrm{tr}(A^\top A) b_j^\top B_S h_n \right|\leq (1-\eta)\lambda,
\end{equation}
where $b_j=Be_j$ and $h_n=\beta_S^\star-\hat{\beta}^\mathrm{det}$.
\end{assumption}

Assumption~\ref{ass:lassoUniqueness} corresponds to the active-set eigenvalue condition used in the PDW framework of~\cite{wainwright2009sharp}. In our setting, it guarantees strong convexity of the deterministic objective and, together with Assumption~\ref{ass:weirdstrong}, ensures that the full lasso solution~\eqref{eq:lassoproblem} is unique with high probability.
In particular, Assumption~\ref{ass:weirdstrong} is a directional version of the strict dual-feasibility condition used in the PDW framework~\cite[Equation~(15)]{wainwright2009sharp}. It controls the mixing of active and inactive coordinates in the particular residual direction $h_n$ and allows us to transfer the universality result from the oracle solution to the full solution.
Assumption~\ref{ass:weirdstrong} can be viewed as a deterministic and directional strict dual-feasibility condition. It is directional and imposes control only on the residual direction $h_n$ which is selected by \eqref{eq:detlasso}. By contrast, the condition in~\cite[Equation~(15)]{wainwright2009sharp} considers the $\ell^\infty\to\ell^\infty$ operator-norm bound $\|X_{S^c}^\top X_S(X_S^\top X_S)^{-1}\|_{\infty\to\infty} \leq (1-\iota)$, where $\iota\in(0,1]$ is the incoherence parameter. Indeed, this controls \emph{all} possible subgradient directions and is therefore stronger than the directional condition used here. Exact signed-support recovery requires this uniform control, whereas the present argument only requires exclusion of false positives.

A pessimistic reader might argue that Assumption~\ref{ass:weirdstrong} is strong. However, we note that it is satisfied by a wide range of dense mixing matrices $A$ and $B$ that satisfy Assumption~\ref{ass:covstruct}. These include orthogonal transforms (Section~\ref{Sec:OrthogonalDesign}), scaled dense Gaussian matrices (Section~\ref{sec:gauss}), and uniformly drawn random orthogonal matrices (Section~\ref{Sec:HaarBoundedSpectrum}).

\subsection{Nontriviality of \eqref{eq:detlasso} via Karush--Kuhn--Tucker (KKT) conditions}\label{sec:nontrivialsolution}

In this section, we characterize precise conditions in which the deterministic solution \eqref{eq:detlasso} is nontrivial, i.e., $\hat{\beta}^\mathrm{det}\neq 0$.
%
Denote the objective function in \eqref{eq:detlasso} as
\begin{equation}\label{eq:determdquation}
F(\beta) \coloneqq \frac{1}{2}\left(\beta_S^\star-\beta\right)^\top Q_n (\beta_S^\star - \beta) + \lambda \|\beta\|_1,
\end{equation}
where $Q_n  \coloneqq  \tau_n B_S^\top B_S$ is positive semidefinite, and $\tau_n \coloneqq \frac{1}{n}\mathrm{tr}\left(A^\top A\right)$. The subdifferential of $F$ is
\begin{equation}
\partial F(\beta)
=
Q_n(\beta-\beta_S^\star)
+\lambda\partial\|\beta\|_1.
\end{equation}
Under Assumption~\ref{ass:lassoUniqueness}, the minimizer of $F$ is unique. Next, at $\beta = 0$, we have $\partial \|0\|_1 = [-1,1]^k$. Thus, the KKT condition that $0\in \partial F(0)$ is equivalent to the existence of a vector $\nu \in [-1,1]^k$ such that $Q_n\beta_S^\star = \lambda \nu$. This is possible if and only if $\|Q_n\beta_S^\star\|_\infty \leq \lambda$. Therefore,
\begin{equation}\label{eq:uniqueKKT}
\hat{\beta}^\mathrm{det}\neq0
\quad\text{if and only if}\quad
\left\|Q_n\beta_S^\star\right\|_\infty>\lambda.
\end{equation}

Note that the presented analysis is well-known (see~\cite{tibshirani2011solution}). Furthermore, as shown in \eqref{eq:uniqueKKT}, the nontriviality of $\hat{\beta}^\mathrm{det}$ is largely determined by $A$ and $B$ as well as the signal form. In the sequel, we show that this nontriviality holds in many cases of interest.

\subsection{Universality Results}

\begin{proposition}\label{prop:RestrictedProp}
Suppose Assumptions~\ref{ass:Sparse}, \ref{ass:covstruct}, and~\ref{ass:lassoUniqueness} hold for fixed $\lambda>0$. Then, for each $W\in\{Z,G\}$ we have
\begin{equation}\label{eq:propFirst}  \|\hat{\beta}^\mathrm{oracle}_S(W) -{\hat{\beta}}^\mathrm{det}\|_2\xrightarrow{\mathbb{P}} 0.
\end{equation}
\end{proposition}

\begin{corollary}\label{cor:Corollary}
As a direct consequence of Proposition~\ref{prop:RestrictedProp}, the following hold
\begin{align}\label{eq:corollary}
\mathbb{E}\left|\|\hat{\beta}^\mathrm{oracle}_S(Z) - \beta^\star_S\|_2^2 - \|\hat{\beta}^\mathrm{oracle}_S(G) - \beta^\star_S\|_2^2\right| &\xrightarrow{} 0,\nonumber\\
\|\hat{\beta}^\mathrm{oracle}_S(Z) - \beta^\star_S\|_2^2 - \|\hat{\beta}^\mathrm{oracle}_S(G) - \beta^\star_S\|_2^2 &\xrightarrow{\mathbb{P}} 0. 
\end{align}
where $G$ is the Gaussian matrix with the same mean and covariance as $Z$. 
\end{corollary}

In Proposition~\ref{prop:RestrictedProp}, we simply claim that the restricted lasso solution converges to the deterministic solution in \eqref{eq:detlasso}, agnostic to the latent variable matrix distribution. Corollary~\ref{cor:Corollary} extends the result to the SE. We use this universality phenomenon as a bridge to extend to the universality of the full estimator \eqref{eq:lassoproblem}. The proofs of Proposition~\ref{prop:RestrictedProp} and consequently Corollary~\ref{cor:Corollary} are discussed in detail in Section~\ref{Sec:SectionIII}. 

\begin{theorem}\label{thm:lassoFullUniversality}
Suppose Assumptions~\ref{ass:Sparse}, \ref{ass:covstruct}, \ref{ass:lassoUniqueness}, and~\ref{ass:weirdstrong} hold. Then, for each $W\in\{Z,G\}$ and fixed $\lambda>0$, with probability tending to one, the lasso has a unique solution and
\begin{equation}\label{eq:Theorem}  \hat{\beta}(W)=\hat{\beta}^\mathrm{oracle}(W).
\end{equation}
\end{theorem}

\begin{corollary}\label{cor:Corollary2}
As a direct consequence of  Corollary~\ref{cor:Corollary} and Theorem~\ref{thm:lassoFullUniversality}, the following hold
\begin{align}\label{eq:Corollary2}
 \mathbb{E}\left|\|\hat{\beta}(Z) - \beta^\star\|_2^2 - \|\hat{\beta}(G) - \beta^\star\|_2^2\right|&\xrightarrow{} 0,\nonumber\\
  \|\hat{\beta}(Z) - \beta^\star\|_2^2 - \|\hat{\beta}(G) - \beta^\star\|_2^2  &\xrightarrow{\mathbb{P}} 0.  
\end{align}
where $G$ is the Gaussian matrix with the same mean and covariance as $Z$. 
\end{corollary}

 Theorem~\ref{thm:lassoFullUniversality} and Corollary~\ref{cor:Corollary2} yield the universality of the lasso solution and the SE in the high-dimensional sparse regime for the linearly dependent covariate model. The proofs are given in Section~\ref{Sec:SectionIV}.

\section{Instantiations of the Universality Results}

In this section, we first verify that all the discussed signals satisfy the assumptions and then analyze the proposed $A$ and $B$ matrix classes.

\subsection{Verification of Assumptions for Example Signal Profiles}

We verify that the signals defined in Remark~\ref{rem:remarkaboutsignals} satisfy Assumption~\ref{ass:Sparse}. Then, we calculate some useful characteristics for each signal. For the geometrically decaying profile \eqref{eq:geomsig} we have
\begin{equation}
\begin{aligned}
\|\beta^\star\|_2^2=\|\beta_S^\star\|_2^2
&=
\alpha^2\frac{1-r^2}{1-r^{2k}}
\sum_{i=1}^k r^{2(i-1)}
=\alpha^2.
\end{aligned}
\end{equation}
Next, we consider the anchor-diffuse profile signal \eqref{eq:anchorsig} and observe that
\begin{align}
\left\lVert \beta^\star \right\rVert_2^2
&=
\alpha^2 \vartheta
+
\alpha^2
\frac{(1-\vartheta)(1-r^2)}
     {1-r^{2(k-1)}}
\sum_{i=2}^{k} r^{2(i-2)}
=
\alpha^2 \vartheta+\alpha^2(1-\vartheta)
=
\alpha^2.
\end{align}

Finally, for the normalized power profile signal \eqref{eq:powersig}, we have
\begin{equation}
\|\beta^\star\|_2^2=\|\beta_S^\star\|_2^2
=
\frac{\alpha^2}{\sum_{j=1}^k j^{-2p}}
\sum_{i=1}^ki^{-2p}
=\alpha^2.
\end{equation}
Thus, when $k=o(n)$, all three signal profiles satisfy Assumption~\ref{ass:Sparse}.


\subsection{Orthogonal Transform Mixing Matrices}\label{Sec:OrthogonalDesign}
As a specific example, we consider the deterministic matrix family of orthogonal transforms and verify that they satisfy the assumptions. Let $B=U_BD_B$ where $U_B$ is orthogonal and $D_B=\mathrm{diag}(d_1,...,d_d)$ with bounded spectrum. On the support set we define $q_\ell \coloneqq \tau_nd_{s_\ell}^2$ for $\ell=1,\dots,k$. Clearly
$B^\top B = D_B^\top U_B^\top U_BD_B=D_B^2$.
Therefore we have $Q_n=\mathrm{diag}(q_1,\dots,q_k)$. Therefore, \eqref{eq:determdquation} becomes
\begin{equation}
F(\beta)
=
\sum_{\ell=1}^k
\left[
\frac{q_\ell}{2}
(\beta_{s_\ell}^\star-\beta_\ell)^2
+\lambda|\beta_\ell|
\right],
\end{equation}
which by soft-thresholding \cite{donoho1995noising}, results in 
\begin{equation}\label{eq:labelIneednow}
\hat{\beta}_\ell^\mathrm{det}=\mathrm{sign}(\beta_{s_\ell}^\star)\left(\left|\beta_{s_\ell}^\star\right|-\frac{\lambda}{q_\ell}\right)_+. 
\end{equation}
By rewriting \eqref{eq:uniqueKKT}, for \eqref{eq:labelIneednow} we obtain
\begin{equation}
\hat{\beta}^\mathrm{det}\neq0
\quad\text{if and only if}\quad
\max_{1\leq \ell\leq k}q_\ell\left|\beta_{s_\ell}^\star\right|>\lambda.
\end{equation}

In the high-dimensional regime, suppose $k\to\infty$ and there is a bounded sequence $(\bar q_\ell)_{\ell\geq1}$ such that $\max_{1\leq\ell\leq k}|q_\ell-\bar q_\ell|\to0$. Then, for all signals defined in Remark~\ref{rem:remarkaboutsignals}, $\max_{1\leq \ell\leq k}q_\ell\left|\beta_{s_\ell}^\star\right|$ converges to the following thresholds:
\begin{itemize}
  \item 
    $\Lambda_\mathrm{geo}=\alpha\sqrt{1-r^2}
\sup_{\ell\geq1}\bar q_\ell r^{\ell-1}$,
\item 
$\Lambda_{\mathrm{ad}}=\alpha\max\left\{\sqrt{\vartheta}\bar q_{1},\sqrt{(1-\vartheta)(1-r^2)}\sup_{\ell\geq 2}\bar{q}_\ell r^{\ell-2}\right\}$,
\item 
$\Lambda_{\mathrm{pow}}=\frac{\alpha}{\sqrt{\sum_{i=1}^\infty i^{-2p}}}
\sup_{\ell\geq1}\bar q_\ell\ell^{-p}$.
\end{itemize}

For each signal, $\lambda<\Lambda$ therefore guarantees $\hat{\beta}^\mathrm{det}\neq0$ for all sufficiently large $n$ (see Section~\ref{sec:nontrivialsolution}). Furthermore, we have that $(B^\top B)_{S^c,S}=0$. Thus, Assumption~\ref{ass:weirdstrong} holds.
%
Assumption~\ref{ass:lassoUniqueness} also holds as soon as $\inf_n \tau_n>0$, and $\inf_n\min_{1\leq \ell\leq k}d_{s_\ell}>0$. Orthogonal transform matrices include many matrices of interest including discrete cosine/sine transforms (DCTs/DSTs) and orthogonal discrete wavelet transforms. These play important roles in various scientific fields such as digital signal processing~\cite{ahmed1974discrete,britanak2010discrete} and in imaging applications like magnetic resonance imaging~\cite{lustig2007sparse}.


\subsection{Conditions to Consider Random  Mixing Matrices}\label{Sec:Transconds}

All findings in this paper are stated for deterministic $A$ and $B$ matrices. The randomness in the framework \eqref{eq:lassoproblem} then stems only from $Z$ and the noise vector $\xi$. Here, we record the conditions under which \emph{random} mixing matrices $A$ and $B$ can be incorporated via a conditioning argument. Informally, we first condition on the matrices and signal, and then apply the deterministic results uniformly over a class on which all constants in the assumptions are controlled.

Formally, consider the tuple $\Theta_n\coloneqq(A,B,S,\beta^\star)$, independent of the latent matrix, its Gaussian counterpart, and the noise vector, and define the $\sigma$-algebra $\mathcal F_n\coloneqq\sigma(\Theta_n)$. Fix constants $C_A,C_B,\kappa>0$ and $\eta\in(0,1)$, and a deterministic sequence $\bar k_n=o(n)$. Let $\mathcal C_n$ be the class of tuples defined by
\begin{equation}
\mathcal C_n \coloneqq
\left\{
(A,B,S,\beta^\star)\ :\ 
\begin{aligned}
&S=\operatorname{supp}(\beta^\star),\quad |S|\leq \bar k_n,\quad
  \|\beta^\star\|_2=\alpha,\\
&\|A\|_{\mathrm{op}}\leq C_A,\quad
  \|B\|_{\mathrm{op}}\leq C_B,\quad
  \lambda_{\min}(Q_n)\geq \kappa,\\
&\max_{j\in S^c}
  \bigl|\tau_n b_j^\top B_S h_n\bigr|
  \leq (1-\eta)\lambda
\end{aligned}
\right\}.
\end{equation}
and define $\mathcal G_n\coloneqq\{\Theta_n\in\mathcal C_n\}$. We discuss the following proposition as the bridge between random and deterministic matrices in the framework.

\begin{proposition}\label{prop:conditional-transfer}
Suppose the signal bounds $\|\beta^\star\|_2=\alpha$ almost surely and $\mathbb P(\mathcal G_n)\to1$. Then the convergence-in-probability conclusions of Proposition~\ref{prop:RestrictedProp} and Theorem~\ref{thm:lassoFullUniversality} hold for random $A$ and $B$. The expectation conclusions of Corollaries~\ref{cor:Corollary} and~\ref{cor:Corollary2} also hold under the additional condition
\begin{equation}\label{eq:preliminary-moment-condition}
\sup_n \mathbb{E}\left[(1+\|A\|_\mathrm{op}+\|B\|_\mathrm{op})^{16}\right]<\infty.
\end{equation}  
\end{proposition}
The proof is given in Appendix~\ref{app:PropTransProof}.


\subsection{Scaled Dense Gaussian Mixing Matrices}\label{sec:gauss}
As a specific example, we consider the random matrix family of scaled dense Gaussian mixing matrices and verify that they satisfy the hypotheses of Proposition~\ref{prop:conditional-transfer}. Let $A$ and $B$ be independent matrices with entries drawn as
\begin{equation} A_{ij}\overset{\mathrm{i.i.d.}}{\sim}\mathcal{N}\left(0,\frac{\zeta}{n}\right),\quad B_{ij}\overset{\mathrm{i.i.d.}}{\sim}\mathcal{N}\left(0,\frac{\omega}{d}\right),
\end{equation}
 where $\zeta,\omega>0$. Suppose $(S,\beta^\star)$ is independent of $(A,B)$. Under these conditions, write $A=\sqrt{\frac{\zeta}{n}}G_A$ and $B_S=\sqrt{\frac{\omega}{d}}G_{B,S}$, where $G_A\in\mathbb{R}^{n\times n}$ and $G_{B,S}\in\mathbb{R}^{d\times k}$ have standard i.i.d.\ Gaussian entries. Then
\begin{equation}
\begin{aligned}
 \tau_n = \frac{1}{n}\|A\|_F^2 = \frac{\zeta}{n^2}\sum_{i,j}(G_A)_{ij}^2,   
\end{aligned}
\end{equation}
where, by the law of large numbers, we can deduce 
\begin{equation}\label{eq:tauLLN}
\tau_n\xrightarrow{\mathbb{P}}\zeta.
\end{equation}
Furthermore, by~\cite[Theorem~4.6.1]{vershynin2018hdp}, we have
\begin{align}\label{eq:opnormLLN}
\left\|B_S^\top B_S-\omega I_k\right\|_\mathrm{op}&= \omega\left\|\frac{1}{d}G_{B,S}^\top G_{B,S}-I_k\right\|_\mathrm{op}
=O_\mathbb{P}\left(\frac{k}{d}+\sqrt{\frac{k}{d}}\right)= o_\mathbb{P}(1).
\end{align}
Therefore, by \eqref{eq:tauLLN}, and \eqref{eq:opnormLLN}, we have $\|Q_n-\zeta\omega I_k\|_\mathrm{op}\xrightarrow{\mathbb{P}}0$, so
$\left\|Q_n\beta_S^\star\right\|_\infty
=
\zeta\omega\left\|\beta_S^\star\right\|_\infty
+o_{\mathbb{P}}(1)$.
This leads to the following thresholds for each signal type:
\begin{itemize}
  \item
         $\Lambda_\mathrm{geo}=\zeta\omega\alpha\sqrt{1-r^2},$
    
\item 
$\Lambda_{\mathrm{ad}}=\zeta\omega\alpha\sqrt{\vartheta},$

\item 
$\Lambda_{\mathrm{pow}}=\zeta\omega\frac{\alpha}{\sqrt{\sum_{i=1}^\infty i^{-2p}}}$ .
\end{itemize}
If $\lambda<\Lambda$, then $\hat{\beta}^\mathrm{det}\neq0$ with probability tending to one. Next, we show that this matrix satisfies Proposition~\ref{prop:conditional-transfer}. Consider the same event $\mathcal{G}_n$ discussed in Section~\ref{Sec:Transconds}. Standard Gaussian operator-norm tail bounds~\cite[Theorem~4.4.1]{vershynin2018hdp} imply \eqref{eq:preliminary-moment-condition}. By Weyl's inequality
(see \cite[Theorem 8.1]{bhatia2007perturbation}), we know $\lambda_\mathrm{min}(Q_n)\xrightarrow{\mathbb{P}} \zeta\omega$. Hence, when $0<\kappa<\zeta \omega$, Assumption~\ref{ass:lassoUniqueness} holds. So, the last step needed to verify $\mathbb{P}(\mathcal{G}_n)\to 1$ which results in the verification of Proposition~\ref{prop:conditional-transfer} for this class of matrices is to show
\begin{equation}\label{eq:eqprovegauss}
\max_{j\in S^c}
\left|\tau_n b_j^\top B_Sh_n\right|
=O_{\mathbb{P}}\left(\sqrt{\frac{\log d}{d}}\right).
\end{equation}
This is proved in Appendix~\ref{app:appB}. Therefore, this class of matrices satisfies Proposition~\ref{prop:conditional-transfer}.

\subsection{Scaled Uniformly Random Orthogonal Mixing Matrices}\label{Sec:HaarBoundedSpectrum}
As a specific example, we consider the random matrix family of scaled uniformly drawn random orthogonal matrices and verify that they satisfy the hypotheses of Proposition~\ref{prop:conditional-transfer}.

Let $A=Q_AD_AQ_A^\top$ and $B=Q_BD_BQ_B^\top$, where $Q_A$ and $Q_B$ are independent uniformly distributed orthogonal matrices and $D_A$ and $D_B$ are deterministic diagonal matrices with uniformly bounded spectra. The spectra of $A^\top A$ and $B^\top B$ are uniformly bounded with empirical spectral distributions converging to $\mu_A$ and $\mu_B$ following Assumption~\ref{def:empspectral}. By the definition of integration against a discrete measure, we have
\begin{equation}
\begin{aligned}
\tau_n=\frac 1 n \sum_{i=1}^n a_i = \int x\,d\left(\frac 1 n \sum_{i=1}^n \delta_{a_i}\right)(x)\longrightarrow\int x\,d\mu_A(x) \eqqcolon \bar{a}.
\end{aligned}
\end{equation}
Similarly, $\bar b_d\coloneqq d^{-1}\operatorname{tr}(B^\top B)\to\bar b\coloneqq\int x\,d\mu_B(x)$. Suppose that $S$ and $\beta^\star$ are independent of $Q_B$. Then, considering the high-dimensional regime, we have
\begin{equation}\label{eq:haar-active-block}
\left\|(B^\top B)_{S,S}-\bar b I_k\right\|_{\mathrm{op}}
\xrightarrow{\mathbb{P}}0.
\end{equation}
This concentration statement is proved in Appendix~\ref{app:appC}.
Similar to Section~\ref{sec:gauss}, we have
$\left\|Q_n\beta_S^\star\right\|_\infty
=
\bar{a}\bar{b}\left\|\beta_S^\star\right\|_\infty
+o_{\mathbb{P}}(1)$,
leaving us with the following thresholds
\begin{itemize}
\item    
$\Lambda_\mathrm{geo}=\bar{a}\bar{b}\alpha\sqrt{1-r^2},$
\item 
$\Lambda_{\mathrm{ad}}=\bar{a}\bar{b}\alpha\sqrt{\vartheta},$
\item 
$\Lambda_{\mathrm{pow}}=\bar{a}\bar{b}\frac{\alpha}{\sqrt{\sum_{i=1}^\infty i^{-2p}}}.
$
\end{itemize}
As in the previous cases, $\lambda<\Lambda$ implies $\hat{\beta}^\mathrm{det}\neq0$ with probability tending to one. Similar to Section~\ref{sec:gauss}, we argue that $A$ and $B$ are bounded in operator norm. Furthermore, by Weyl's inequality, $\lambda_\mathrm{min}(Q_n)\xrightarrow{\mathbb{P}}\bar a\bar b$; hence Assumption~\ref{ass:lassoUniqueness} holds with high probability whenever $\bar a\bar b>0$. Finally, in Appendix~\ref{app:appC}, we show that
\begin{equation}\label{eq:proveHaarGood}
\max_{j\in S^c}
\left|\tau_n b_j^\top B_Sh_n\right|
=O_{\mathbb{P}}\left(\sqrt{\frac{k+\log d}{d}}\right)
\end{equation}
holds, which implies that this class of matrices satisfies Proposition~\ref{prop:conditional-transfer}.

\section{Proofs of Proposition~\ref{prop:RestrictedProp} and Corollary~\ref{cor:Corollary}}\label{Sec:SectionIII}
In this section, we state and prove the intermediary results needed to prove Proposition~\ref{prop:RestrictedProp} and Corollary~\ref{cor:Corollary}. Figure~\ref{fig:restricted-dependencies} is an illustrative summary of the dependencies of this section.

\begin{figure*}[!htb]
	\centering
	\begin{tikzpicture}[
		x=1cm,y=1cm,
		>={Stealth[length=1.8mm,width=1.25mm]},
		every node/.style={font=\sffamily\scriptsize},
		assumptionbox/.style={draw=black!56,dashed,line width=0.5pt,rounded corners=1pt,fill=black!2,align=center,inner xsep=4pt,inner ysep=3pt,text width=2.45cm,minimum height=0.78cm},
		lemmabox/.style={draw=black!64,line width=0.5pt,rounded corners=1pt,fill=black!1,align=center,inner xsep=4pt,inner ysep=3pt,text width=2.15cm,minimum height=0.78cm},
		resultbox/.style={draw=blue!52!black,line width=0.6pt,rounded corners=1pt,fill=blue!6,align=center,inner xsep=4pt,inner ysep=4pt,text width=4.25cm,minimum height=1.08cm},
		finalbox/.style={resultbox,fill=blue!11},
		entry/.style={->,draw=black!68,line width=0.55pt,line cap=round,line join=round,shorten <=0.6pt,shorten >=0pt},
		transition/.style={->,draw=black!68,line width=0.55pt,line cap=round,line join=round,shorten <=0pt,shorten >=0pt},
		merge/.style={draw=black!65,line width=0.5pt,line cap=round,line join=round,shorten <=0.6pt,shorten >=0.6pt},
		feed/.style={merge,shorten <=0pt,shorten >=0.6pt},
		standing/.style={->,draw=black!55,dashed,line width=0.5pt,line cap=round,line join=round,shorten <=0pt,shorten >=0pt},
		standingfeed/.style={draw=black!53,dashed,line width=0.5pt,line cap=round,line join=round,shorten <=0pt,shorten >=0.6pt},
		junction/.style={circle,draw=black!68,fill=black!68,line width=0.3pt,inner sep=0pt,minimum size=2.4pt}
		]
		\node[assumptionbox] (baseass) at (-1.05,1.25)
		{\textbf{Assumptions~\ref{ass:Sparse}--\ref{ass:covstruct}}};
		\node[assumptionbox] (curvature) at (0.35,0)
		{\textbf{Assumption~\ref{ass:lassoUniqueness}}};
		
		\node[lemmabox] (moment) at (-4.05,1.25)
		{\textbf{Lemma~\ref{lem:moment}}};
		\node[lemmabox] (bilinear) at (-2.45,0)
		{\textbf{Lemma~\ref{lem:bilinear}}};
		\node[lemmabox] (active) at (0.35,-1.20)
		{\textbf{Lemma~\ref{lem:active-conc}}};
		\node[lemmabox] (stability) at (3.15,-1.20)
		{\textbf{Lemma~\ref{lem:Lemma14}}};
		
		\node[resultbox] (proposition) at (0.35,-2.60)
		{\textbf{Proposition~\ref{prop:RestrictedProp}}\\
			oracle lasso converges to the deterministic lasso};
		\node[finalbox] (correstricted) at (0.35,-4.05)
		{\textbf{Corollary~\ref{cor:Corollary}}\\
			oracle SE universality};
		
		\node[junction] (activejoin) at (0.35,-0.64) {};
		\node[junction] (propjoin) at ($(proposition.north)+(0,0.18)$) {};
		\node[junction] (restrictedjoin) at ($(correstricted.north)+(0,0.18)$) {};
		\node[junction] (momentbranch) at ($(moment.south)+(0,-0.20)$) {};
		
		\draw[standing] (baseass.west) -- (moment.east);
		\draw[standing] (baseass.south west) -- (bilinear.north);
		\draw[standingfeed] (curvature.south) -- (activejoin);
		\draw[feed] (bilinear.south)
		.. controls (-2.45,-0.61) and (-0.75,-0.64) .. (activejoin);
		\draw[entry] (activejoin) -- (active.north);
		\draw[transition] (active.east) to[out=0,in=180] (stability.west);
		
		\draw[feed] (moment.south) -- (momentbranch);
		\draw[merge] (momentbranch)
		.. controls (-3.90,-0.75) and (-2.40,-1.88) .. (propjoin);
		\draw[feed] (active.south) -- (propjoin);
		\draw[feed] (stability.south)
		.. controls (3.15,-1.80) and (1.45,-1.88) .. (propjoin);
		\draw[entry] (propjoin) -- (proposition.north);
		
		\draw[feed] (proposition.south) -- (restrictedjoin);
		\draw[merge] (momentbranch)
		.. controls (-4.25,-1.85) and (-3.10,-3.33) .. (-2.05,-3.33)
		-- (restrictedjoin);
		\draw[entry] (restrictedjoin) -- (correstricted.north);
	\end{tikzpicture}
	\caption{Proposition~\ref{prop:RestrictedProp} and Corollary~\ref{cor:Corollary}. Dashed arrows and boxes indicate assumptions, and solid arrows indicate direct proof dependencies.}
	\label{fig:restricted-dependencies}
\end{figure*}
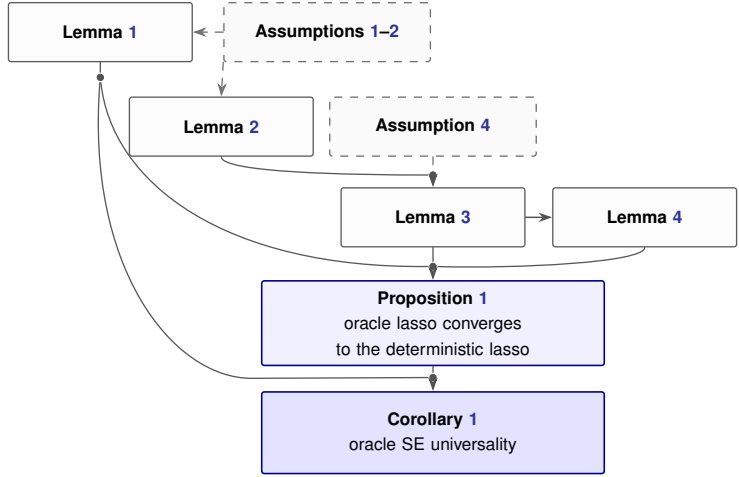

\subsection{Auxiliary Lemmas}
The first lemma we present shows $\frac{\|y(W)\|_2^2}{n}$ has a uniformly bounded fourth moment, which implies stochastic boundedness. Furthermore, the lemma shows that the SE of the full and oracle estimators are bounded in second moment, which implies uniform integrability. This is motivated by our need to compare the solutions of \eqref{eq:lassoproblem} and \eqref{eq:OracleObj} which carry randomness within the deterministic solution \eqref{eq:detlasso} as stated in Proposition~\ref{prop:RestrictedProp}. Thus, we present the first lemma as follows. 

\begin{lemma}\label{lem:moment}
Suppose that Assumptions~\ref{ass:Sparse}--\ref{ass:covstruct} hold. Then
\begin{equation}\label{eq:bound1}
    \sup_n \mathbb{E}\left(\frac{\|{y(W)\|_2^2}}{n}\right)^4<\infty,
\end{equation}
\begin{equation}\label{eq:bound2}
\sup_n \mathbb{E}\left[\|\hat{\beta}-\beta^\star\|_2^4+\|\hat{\beta}^\mathrm{oracle}_S-\beta^\star_S\|_2^4\right]<\infty.
\end{equation}
for all $W\in\{Z,G\}$.
\end{lemma}
The proof is given in Appendix~\ref{app:appD}. In the next lemma, we introduce two bounds that are directly used in formulating the population lasso problem \eqref{eq:detlasso}, these inequalities are directly used in a future lemma to prove Proposition~\ref{prop:RestrictedProp}. More specifically, this lemma provides the necessary tool to draw equivalence between the random \eqref{eq:OracleObj} and deterministic \eqref{eq:detlasso} lasso problem  by means of future lemmas.

\begin{lemma}\label{lem:bilinear}
Let $W\in\{Z,G\}$, and let $p,q\in\mathbb{R}^d$ be deterministic vectors that satisfy $\|p\|_2,\|q\|_2\leq L$ where $L<\infty$ is a constant. Then, there exist constants $c,C>0$ such that for all $t>0$ the following inequalities hold. 
\begin{equation}
    \mathbb{P}\left(
    \left |{\frac{1}{n} p^\top W^\top A^\top AWq-\tau_np^\top q}\right|>t
    \right)
    \le
    C\mathrm{e}^{-cn\min(t^2,t)}.
\label{eq:probinequal1}
\end{equation}

\begin{equation}\label{eq:probinequal2}
    \mathbb{P}\left(
        \left.
        \left|
            \frac{1}{n} p^\top W^\top A^\top \xi
        \right| > t
        \,\right|\, \xi
    \right)
    \leq
    2 \mathrm{e}^{-\frac{c n^2 t^2}{\| \xi \|_2^2}},
\end{equation}
where we recall that $\tau_n=\frac 1 n \mathrm{tr}(A^\top A)$.
\end{lemma}
The proof is given in Appendix~\ref{app:appE}. Next, we use the inequalities proved in Lemma~\ref{lem:bilinear} to show the equivalence of \eqref{eq:OracleObj} and \eqref{eq:detlasso} with respect to the active coordinates in our settings.
\begin{lemma}
\label{lem:active-conc}
Suppose Assumptions~\ref{ass:Sparse}, \ref{ass:covstruct}, and~\ref{ass:lassoUniqueness} hold. Then, for each $W\in\{Z,G\}$, the following hold.
\begin{equation}\label{eq:ineq1lemm6}
    \norm{\frac1n X_S(W)^\top X_S(W)-\tau_nB_S^\top B_S}_{\mathrm{op}}
    \xrightarrow{\mathbb{P}}0,
\end{equation}
\begin{equation}
\label{eq:ineq2lemm6}
    \norm{\frac1n X_S(W)^\top\xi}_2\xrightarrow{\mathbb{P}}0.
\end{equation}
\end{lemma}
The proof is given in Appendix~\ref{app:appF}. The next lemma is the final ingredient needed in this section. It establishes that the lasso under Assumption~\ref{ass:lassoUniqueness} is strongly convex with high probability, with a stable solution. This is the final fact needed to prove Proposition~\ref{prop:RestrictedProp} and Corollary~\ref{cor:Corollary}.

\begin{lemma}\label{lem:Lemma14}
Suppose Assumptions~\ref{ass:Sparse}, \ref{ass:covstruct}, and~\ref{ass:lassoUniqueness} hold. Then, $\sup_n \|\hat{\beta}^\mathrm{det}\|_2^2<\infty$, and for all $W\in\{Z,G\}$, we have
\begin{equation}\label{eq:stufftoproveforuniqueness}
\mathbb{P}\left(\lambda_\mathrm{min}\left(\frac 1 n X_S(W)^\top X_S(W)\right)\geq \frac \kappa 2\right)\to 1.
\end{equation}
\end{lemma}
The proof is given in Appendix~\ref{app:appG}.


\subsection{Proof of Proposition~\ref{prop:RestrictedProp}} \label{subsec:Gngn}


\begin{proof}
Note that throughout this proof we use the same $F(\beta)$ notation as in \eqref{eq:determdquation}, and denote the objective function in \eqref{eq:OracleObj} as
\begin{equation}
F(\beta,W) \coloneqq \frac1{2n}\|y(W)-X_S(W)\beta\|_2^2+\lambda\|\beta\|_1.
\end{equation}
By \eqref{eq:optimalbound} in Appendix~\ref{app:appD} and properties of the $\ell^1$- and $\ell^2$- norm, we see that
$\|\hat{\beta}^\mathrm{oracle}_S\|_2\leq Y_n(W) / (2\lambda)$,
where $Y_n(W) \coloneqq 
{\|y(W)\|_2^2} / {n}$, and $Y_n(W)$ is $O_\mathbb
{P}(1)$. This means that $\hat{\beta}^\mathrm{oracle}$ is also stochastically bounded. For any $\delta>0$, choose $M_\delta<\infty$ large enough that the set $\mathcal K_\delta\coloneqq\{\beta\in\mathbb R^k:\|\beta\|_2\leq M_\delta\}$ contains $\hat\beta^\mathrm{det}$ for every $n$, as guaranteed by Lemma~\ref{lem:Lemma14}, and satisfies
$\mathbb P(\hat\beta_S^\mathrm{oracle}\in\mathcal K_\delta)\geq1-\delta$ for all sufficiently large $n$.
Thus, on this event, both the oracle and deterministic lasso solutions lie in $\mathcal K_\delta$.
Observe from \eqref{eq:OracleObj} that the noise term $\frac{1}{2n}\|\xi\|_2^2$ has no effect on the minimizer. Thus, we define $\Bar{F}(\beta,W) = F(\beta,W) - \frac
1{2n} \|\xi\|_2^2$, and through some algebraic manipulations obtain  
\begin{align}
\sup_{\beta\in\mathcal{K}_\delta}\left|\Bar{F}(\beta,W)-F(\beta)\right|&\leq \frac{1}{2}(\alpha+M_\delta)^2\|G_n(W)\|_\mathrm{op}+(\alpha+M_\delta)\|g_n(W)\|_2,
\end{align}
where $G_n(W) \coloneqq \frac1n X_S(W)^\top X_S(W)-\tau_nB_S^\top B_S$ and $g_n(W) \coloneqq \frac1n X_S(W)^\top \xi$.
By Lemma~\ref{lem:active-conc}, we have
$\Delta_{n,\delta}(W)\xrightarrow{\mathbb{P}}0$, where $\Delta_{n,\delta}(W) \coloneqq \sup_{\beta\in\mathcal{K}_\delta}\left|\Bar{F}(\beta,W)-F(\beta)\right|$. 

Now, because $F(\beta)$ is $ \kappa $-strongly convex, we deduce that
\begin{equation}
\frac \kappa 2 \|\hat{\beta}^\mathrm{oracle}_S-{\hat{\beta}}^\mathrm{det}\|_2^2\leq F(\hat{\beta}^\mathrm{oracle}_S)-F(\hat{\beta}^\mathrm{det}).
\end{equation}
On the event $\{\hat\beta_S^\mathrm{oracle}\in\mathcal K_\delta\}$, optimality of the two minimizers gives $F(\hat{\beta}^\mathrm{oracle}_S)-F(\hat{\beta}^\mathrm{det})\leq2\Delta_{n,\delta}(W)$, and hence
\begin{equation}\label{eq:eq94}
\|\hat{\beta}^\mathrm{oracle}_S-{\hat{\beta}}^\mathrm{det}\|_2^2 \leq \frac 4 \kappa \Delta_{n,\delta}(W).
\end{equation}
Therefore, for every $\epsilon>0$,
\begin{equation}
\mathbb P\left(\|\hat\beta_S^\mathrm{oracle}-\hat\beta^\mathrm{det}\|_2>\epsilon\right)
\leq\delta+\mathbb P\left(\Delta_{n,\delta}(W)>\frac{\kappa\epsilon^2}{4}\right).
\end{equation}
Taking $n\to\infty$ and then $\delta\downarrow0$ proves \eqref{eq:propFirst}.
\end{proof}

\subsection{Proof of Corollary~\ref{cor:Corollary}}\label{sec:ProofCor9}
\begin{proof}
By Proposition~\ref{prop:RestrictedProp}, Lemma~\ref{lem:Lemma14}, and the stochastic boundedness of $\hat{\beta}^\mathrm{oracle}_S$ established by \eqref{eq:optimalbound}, we have
\begin{equation}
\left|\big\|\hat{\beta}^\mathrm{oracle}_S-\beta^\star_S\big\|_2^2-\big\|\hat{\beta}^\mathrm{det}-\beta^\star_S\big\|_2^2\right|\xrightarrow{\mathbb{P}}0.
\end{equation}
We use uniform integrability implied by Lemma~\ref{lem:moment} and Vitali's theorem \cite{billingsley1995probability} to obtain
\begin{equation}
\mathbb{E}\left|\big\|\hat{\beta}^\mathrm{oracle}_S-\beta^\star_S\big\|_2^2-\big\|{\hat{\beta}}^\mathrm{det}-\beta^\star_S\big\|_2^2\right|\xrightarrow{}0.
\end{equation}
By the triangle inequality, we have
\begin{equation}
\mathbb{E}\left|\big\|\hat{\beta}^\mathrm{oracle}_S(Z) - \beta^\star_S\big\|_2^2 - \big\|\hat{\beta}^\mathrm{oracle}_S(G) - \beta^\star_S\big\|_2^2\right| \to 0.
\end{equation}
Since convergence in mean implies convergence in probability, both statements in \eqref{eq:corollary} follow. 
\end{proof}


\section{Proofs of Theorem~\ref{thm:lassoFullUniversality} and Corollary~\ref{cor:Corollary2}}\label{Sec:SectionIV}

In this section, we state the lemma necessary to complete the proof of Theorem~\ref{thm:lassoFullUniversality}. This lemma shows that the full and oracle lasso problems are equivalent. Figure~\ref{fig:full-dependencies} summarizes the result dependencies of this section.

	\begin{figure}[!htb]
		\centering
		\begin{tikzpicture}[
			x=1cm,y=1cm,
			>={Stealth[length=1.8mm,width=1.25mm]},
			every node/.style={font=\sffamily\scriptsize},
			assumptionbox/.style={draw=black!56,dashed,line width=0.5pt,rounded corners=1pt,fill=black!2,align=center,inner xsep=4pt,inner ysep=3pt,text width=2.90cm,minimum height=0.78cm},
			lemmabox/.style={draw=black!64,line width=0.5pt,rounded corners=1pt,fill=black!1,align=center,inner xsep=4pt,inner ysep=3pt,text width=2.15cm,minimum height=0.78cm},
			resultbox/.style={draw=blue!52!black,line width=0.6pt,rounded corners=1pt,fill=blue!6,align=center,inner xsep=4pt,inner ysep=4pt,text width=4.25cm,minimum height=1.08cm},
			finalbox/.style={resultbox,fill=blue!11},
			entry/.style={->,draw=black!68,line width=0.55pt,line cap=round,line join=round,shorten <=0.6pt,shorten >=0pt},
			transition/.style={->,draw=black!68,line width=0.55pt,line cap=round,line join=round,shorten <=0pt,shorten >=0pt},
			merge/.style={draw=black!65,line width=0.5pt,line cap=round,line join=round,shorten <=0.6pt,shorten >=0.6pt},
			feed/.style={merge,shorten <=0pt,shorten >=0.6pt},
			standing/.style={->,draw=black!55,dashed,line width=0.5pt,line cap=round,line join=round,shorten <=0pt,shorten >=0pt},
			standingfeed/.style={draw=black!53,dashed,line width=0.5pt,line cap=round,line join=round,shorten <=0pt,shorten >=0.6pt},
			junction/.style={circle,draw=black!68,fill=black!68,line width=0.3pt,inner sep=0pt,minimum size=2.4pt}
			]
			\node[assumptionbox] (standingass) at (-2.45,1.20)
			{\textbf{Assumptions~\ref{ass:Sparse}, \ref{ass:covstruct}, \ref{ass:lassoUniqueness}, \ref{ass:weirdstrong}}};
			\node[lemmabox] (inactive) at (0.65,1.20)
			{\textbf{Lemma~\ref{lem:inactive-conc}}};
			\node[lemmabox] (stability) at (0.65,0)
			{\textbf{Lemma~\ref{lem:Lemma14}}};
			\node[resultbox] (proposition) at (4.35,0)
			{\textbf{Proposition~\ref{prop:RestrictedProp}}\\
				oracle-to-deterministic convergence};
			
			\node[resultbox] (theoremfull) at (0.65,-1.25)
			{\textbf{Theorem~\ref{thm:lassoFullUniversality}}\\
				strict dual feasibility gives $\hat\beta=\hat\beta^{\mathrm{oracle}}$ w.h.p.};
			\node[resultbox] (correstricted) at (-4.15,-1.25)
			{\textbf{Corollary~\ref{cor:Corollary}}\\
				oracle SE universality};
			\node[finalbox] (corfull) at (0.65,-2.65)
			{\textbf{Corollary~\ref{cor:Corollary2}}\\
				full lasso SE universality};
			
			\node[junction] (theoremjoin) at (0.65,-0.60) {};
			\node[junction] (fulljoin) at ($(corfull.north)+(0,0.12)$) {};
			
			\draw[standing] (standingass.east) -- (inactive.west);
			\draw[feed] (inactive.south west)
			.. controls (-0.90,0.72) and (-0.95,-0.25) .. (-0.95,-0.40)
			.. controls (-0.95,-0.52) and (-0.15,-0.60) .. (theoremjoin);
			\draw[feed] (stability.south) -- (theoremjoin);
			\draw[feed] (proposition.south)
			.. controls (4.35,-0.58) and (1.90,-0.60) .. (theoremjoin);
			\draw[entry] (theoremjoin) -- (theoremfull.north);
			
			\draw[feed] (theoremfull.south) -- (fulljoin);
			\draw[feed] (correstricted.south)
			.. controls (-4.15,-1.98) and (-1.50,-1.99) .. (fulljoin);
			\draw[entry] (fulljoin) -- (corfull.north);
		\end{tikzpicture}
		\caption{Dependency structure for Theorem~\ref{thm:lassoFullUniversality} and Corollary~\ref{cor:Corollary2}. Dashed arrows and boxes indicate assumptions, and solid arrows indicate direct proof dependencies.}
		\label{fig:full-dependencies}
	\end{figure}
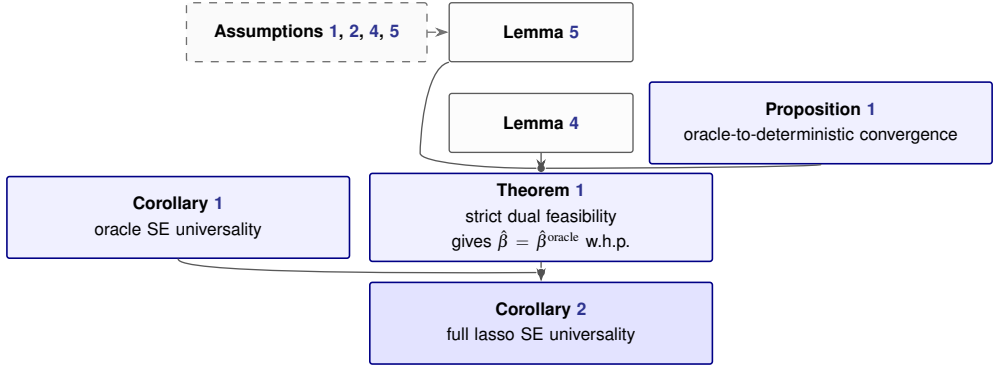
    
\begin{lemma}
\label{lem:inactive-conc}
Suppose Assumptions~\ref{ass:Sparse}, \ref{ass:covstruct}, \ref{ass:lassoUniqueness}, and~\ref{ass:weirdstrong} hold. Then, for each $W\in\{Z,G\}$,
\begin{equation}
    \max_{j\in S^c}
    \abs{\frac1n X_j(W)^\top X_S(W)h_n-\tau_nb_j^\top B_Sh_n}
    \xrightarrow{\mathbb{P}}0,
\label{eq:inactive-signal-conc}
\end{equation}
\begin{equation}
    \max_{j\in S^c}
    \norm{\frac1n X_j(W)^\top X_S(W)}_2
    =O_\mathbb{P}(1),
    \label{eq:inactive-cross-op}
\end{equation}
and
\begin{equation}
    \max_{j\in S^c}
    \abs{\frac1n X_j(W)^\top\xi}
    \xrightarrow{\mathbb{P}}0.
    \label{eq:inactive-noise-conc}
\end{equation}
\end{lemma}
The proof is given in Appendix~\ref{app:appH}.
\subsection{Proof of Theorem~\ref{thm:lassoFullUniversality}}

\begin{proof}
Let $h(W)
=\beta_S^\star
-\hat{\beta}_S^\mathrm{oracle}$, then we can rewrite $r_S(W)=X_S(W)h(W)+\xi
$. For all inactive coordinates $j\in S^c$, we have
\begin{align}
&\phantom{={}}\frac{1}{n}X_j(W)^\top r_S(W)
-\tau_n b_j^\top B_S h_n \nonumber\\
&=
\left[
\frac{1}{n}X_j(W)^\top X_S(W)h_n
-\tau_n b_j^\top B_S h_n
\right]
+
\frac{1}{n}X_j(W)^\top X_S(W)
\left(h(W)-h_n\right)
+
\frac{1}{n}X_j(W)^\top\xi.
\end{align}
By taking the absolute values and maximizing, we get
\begin{align}
\max_{j\in S^c}
\left|
\frac{1}{n}X_j(W)^\top r_S(W)
-
\tau_n b_j^\top B_S h_n
\right|&\leq\max_{j\in S^c}
\left|
\frac{1}{n}X_j(W)^\top X_S(W)h_n
-
\tau_n b_j^\top B_S h_n
\right|\nonumber\\
&\quad+\max_{j\in S^c}
\left|
\frac{1}{n}X_j(W)^\top X_S(W)
\left(h(W)-h_n\right)\right|\nonumber\\
&\quad+\max_{j\in S^c}
\left|
\frac{1}{n}X_j(W)^\top\xi
\right|.
\end{align}
The first and third terms converge to 0 by Lemma~\ref{lem:inactive-conc}. For the middle term, by Proposition~\ref{prop:RestrictedProp}, and by Lemma~\ref{lem:inactive-conc} this term also converges to 0, so

\begin{equation}\label{eq:equationIneed}
\max_{j\in S^c}
\left|
\frac{1}{n}X_j(W)^\top r_S(W)
-
\tau_n b_j^\top B_S h_n
\right|
\xrightarrow{\mathbb{P}}0.
\end{equation}
Let $D_n(W)\coloneqq \max_{j\in S^c}
\left|
\frac{1}{n}X_j(W)^\top r_S(W)
-
\tau_n b_j^\top B_S h_n
\right|$, for all $j\in S^c$, we have
\begin{align}
\left|
\frac{1}{n}X_j(W)^\top r_S(W)
\right|
\leq
\left|\tau_n b_j^\top B_S h_n\right|
+
\left|
\frac{1}{n}X_j(W)^\top r_S(W)
-\tau_n b_j^\top B_S h_n
\right| .
\end{align}
By maximizing over the inactive coordinates we have
\begin{equation}
\left\|
\frac{1}{n}X_{S^c}(W)^\top r_S(W)
\right\|_\infty
\leq
(1-\eta)\lambda
+
D_n(W).
\end{equation}
Thus, on the event $D_n(W)<\eta\lambda$, which has probability tending to one by \eqref{eq:equationIneed}, we have $\left\|
\frac{1}{n}X_{S^c}(W)^\top r_S(W)
\right\|_\infty
<
\lambda$, therefore
\begin{equation}\label{eq:eq113}
\mathbb{P}\left(
\left\|
\frac{1}{n}X_{S^c}(W)^\top r_S(W)
\right\|_\infty
<
\lambda
\right)
\to 1.
\end{equation}
The oracle lasso problem \eqref{eq:OracleObj} minimizes the restricted objective over coordinates in \(S\). Therefore, its active coordinates satisfy the KKT condition
\begin{equation}
\begin{aligned}
0
&\in
-\frac{1}{n}X_S(W)^\top r_S(W)
+\lambda
\partial\big\|\hat{\beta}^\mathrm{oracle}_S(W)\big\|_1.
\end{aligned}
\end{equation}
This is equivalent to the existence of $z_S\in \partial\|\hat{\beta}^\mathrm{oracle}_S(W)\|_1$ that satisfies
\begin{equation}
\begin{aligned}
0
&=-\frac{1}{n}X_S(W)^\top r_S(W)
+\lambda z_S .
\end{aligned}
\end{equation}
On the event in \eqref{eq:eq113}, let $z_{S^c}=\frac{1}{n\lambda}X_{S^c}(W)^\top r_S(W)$. Then $\|z_{S^c}\|_\infty<1$, and since $\hat{\beta}_{S^c}^\mathrm{oracle}(W)=0$, we have $z_{S^c}\in\partial\big\|\hat{\beta}_{S^c}^\mathrm{oracle}(W)\big\|_1$. Hence, let $z=(z_S,z_{S^c})$, then $z\in\partial\big\|\hat{\beta}^\mathrm{oracle}(W)\big\|_1$.

Considering the full lasso problem \eqref{eq:lassoproblem}, we have the following KKT conditions
\begin{equation}
\begin{aligned}
0
&\in
-\frac{1}{n}X(W)^\top\bigl(y(W)-X(W)\beta\bigr)+\lambda \partial\|\beta\|_1 ,
\end{aligned}
\end{equation}
 We validate this condition for  $\hat{\beta}^\mathrm{oracle}$, with $r_\mathrm{full}(W)=y(W)-X(W)\hat{\beta}^\mathrm{oracle}(W)$. Considering the active coordinates, 
\begin{equation}
0
=
-
\frac{1}{n}X_S(W)^\top r_\mathrm{full}(W)
+
\lambda z_S.
\end{equation}
For the inactive coordinates, since $\hat{\beta}_{S^c}^\mathrm{oracle} = 0$, also $\lambda  z_{S^c} = \frac{1}{n}X_{S^c}(W)^\top r_\mathrm{full}(W)$, therefore
\begin{equation}
\begin{aligned}
&-
\frac{1}{n}X_{S^c}(W)^\top r_\mathrm{full}(W)
+
\lambda z_{S^c}=0.
\end{aligned}
\end{equation}
Thus, the full vector KKT condition holds
\begin{equation}
0
\in
-
\frac{1}{n}X(W)^\top r_\mathrm{full}(W)
+
\lambda\,
\partial\big\|\hat{\beta}^\mathrm{oracle}(W)\big\|_1 .
\end{equation}
This means that $\hat{\beta}^\mathrm{oracle}(W)$ is a full lasso minimizer. Through some algebraic steps, we show every full lasso minimizer in these settings is supported on $S$ and is unique. Assume $\tilde{\beta}$ is another candidate for a full lasso minimizer, and let $\Delta = \tilde{\beta}-\hat{\beta}^\mathrm{oracle}$, together with $r = y-X\hat{\beta}^\mathrm{oracle}$. We further denote the objective function of \eqref{eq:lassoproblem} as
\begin{equation}
F_\mathrm{full}(\beta) \coloneqq \frac 1{2n} \left\|y-X\beta\right\|_2^2+\lambda\|\beta\|_1.
\end{equation}
The KKT conditions provide $\frac 1 n X^\top r = \lambda z$. Furthermore
\begin{align}\label{eq:equationtoprovefinaltheorem}
F_\mathrm{full}(\hat{\beta}^\mathrm{oracle}+\Delta)-F_\mathrm{full}(\hat{\beta}^\mathrm{oracle})=\frac{1}{2n}\|X\Delta\|_2^2
+ \lambda \left(\|\hat{\beta}^\mathrm{oracle}+\Delta\|_1-\|\hat{\beta}^\mathrm{oracle}\|_1-z^\top \Delta\right).
\end{align}

Since $\tilde\beta$ is also a minimizer, the left-hand side of \eqref{eq:equationtoprovefinaltheorem} is $0$. Using the subgradient inequality on $S$, together with $\hat\beta_{S^c}^\mathrm{oracle}=0$ and $\|z_{S^c}\|_\infty<1$, gives
\begin{equation}\label{eq:finalequation}
0\geq \frac{1}{2n}\|X\Delta\|_2^2+\lambda \left(1-\|z_{S^c}\|_\infty\right)\|\Delta_{S^c}\|_1\geq 0.
\end{equation}
Both terms in \eqref{eq:finalequation} must therefore vanish. It follows that $\Delta_{S^c}=0$. On the event in Lemma~\ref{lem:Lemma14} where $X_S$ has full column rank, $X\Delta=X_S\Delta_S=0$ then implies $\Delta_S=0$. Thus, on the intersection of these two events, $\hat\beta^\mathrm{oracle}$ is the unique full lasso minimizer. Both events have probability tending to one, which proves the theorem.
\end{proof}

\subsection{Proof of Corollary~\ref{cor:Corollary2}}

\begin{proof}
In Theorem~\ref{thm:lassoFullUniversality}, we proved that, with probability tending to one, the lasso solution is unique and equals the solution of \eqref{eq:OracleObj}. Suppose we call the event on which Theorem~\ref{thm:lassoFullUniversality} holds $\mathcal{E}_{\mathrm{thm}}$. From here, we have
\begin{align}
\left|\big\|\hat{\beta}(W)-\beta^\star\big\|_2^2-\big\|\hat{\beta}^\mathrm{oracle}(W)-\beta^\star\big\|_2^2\right|
= \left|\big\|\hat{\beta}(W)-\beta^\star\big\|_2^2-\big\|\hat{\beta}^\mathrm{oracle}(W)-\beta^\star\big\|_2^2\right|\mathbf{1}_{\mathcal{E}^c_{\mathrm{thm}}}.
\end{align}
By invoking the Cauchy-Schwarz inequality and taking expectations we have
\begin{align}
\mathbb{E}\left|\big\|\hat{\beta}(W)-\beta^\star\big\|_2^2-\big\|\hat{\beta}^\mathrm{oracle}(W)-\beta^\star\big\|_2^2\right|
&\leq \left[\mathbb{E}\left(\big\|\hat{\beta}(W)-\beta^\star\big\|_2^2-\big\|\hat{\beta}^\mathrm{oracle}(W)-\beta^\star\big\|_2^2\right)^2\right]^\frac{1}{2}\mathbb{P}\left(\mathcal{E}^c_{\mathrm{thm}}\right)^\frac{1}{2} \nonumber \\
&\to 0,
\end{align}
 by Lemma~\ref{lem:moment} and Theorem~\ref{thm:lassoFullUniversality}. The final step follows from the triangle inequality and yields
 \begin{equation}   \mathbb{E}\left|\big\|\hat{\beta}(Z) - \beta^\star\big\|_2^2 - \big\|\hat{\beta}(G) - \beta^\star\big\|_2^2\right|\xrightarrow{} 0,
 \end{equation}
where Markov's inequality yields both statements in \eqref{eq:Corollary2}.
\end{proof}

\section{Numerical Illustrations}\label{Sec:SectionV}

In this section, we provide some experiments to verify the theoretical claims made throughout this paper. We empirically show that the SE of the full and oracle estimators are asymptotically equivalent in settings satisfying the paper's assumptions.

\begin{figure}[!htb]
\centering
\begin{subfigure}[t]{0.31\linewidth}
    \includegraphics[width=\linewidth]{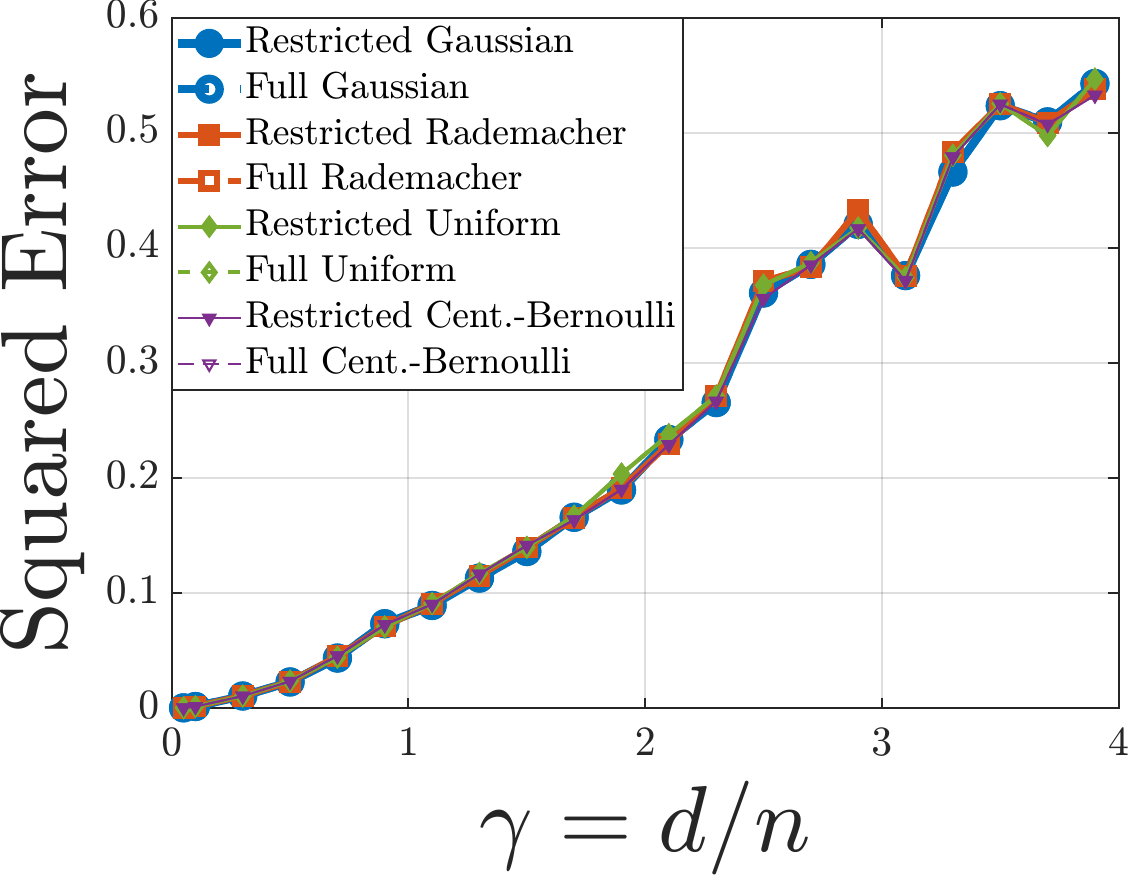}
    \caption{Geometric profile ($r=0.25$)}
\end{subfigure}
\hfill
\begin{subfigure}[t]{0.31\linewidth}
    \includegraphics[width=\linewidth]{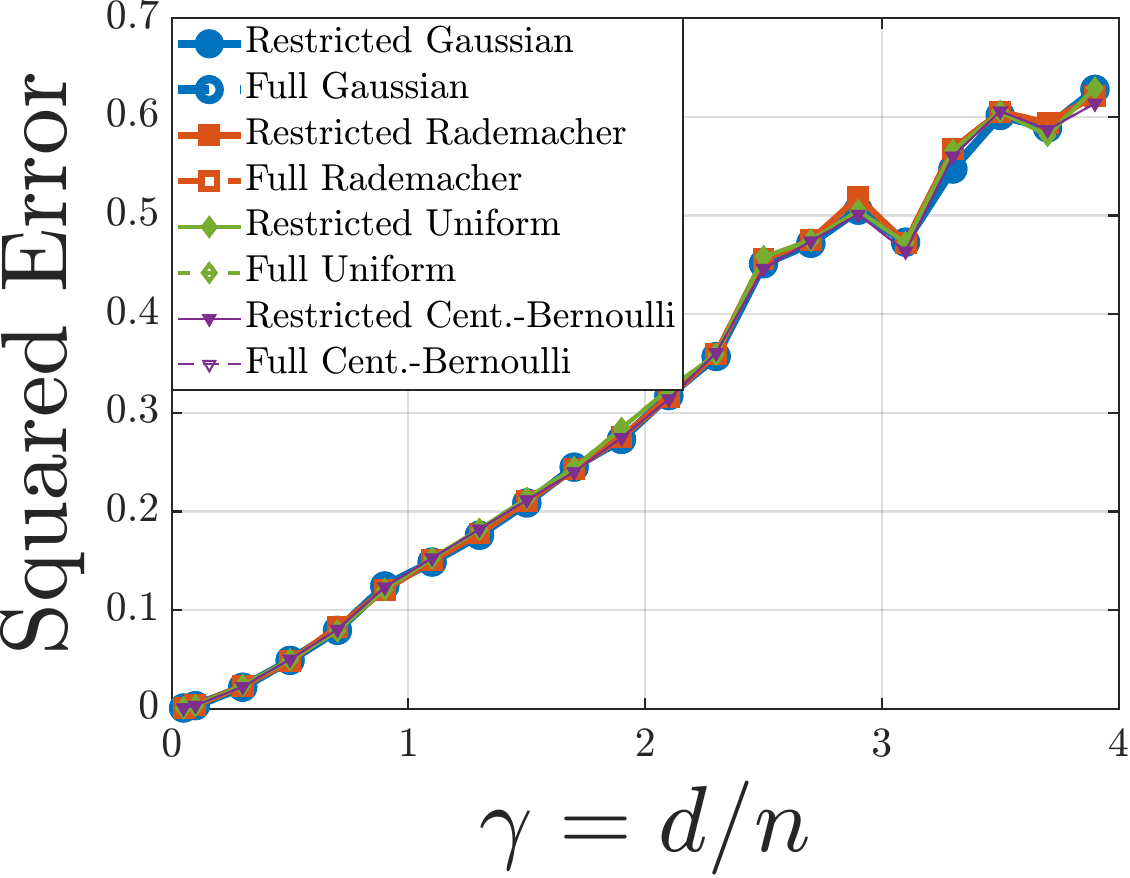}
    \caption{Power profile ($p=1.5$)}
\end{subfigure}
\hfill
\begin{subfigure}[t]{0.31\linewidth}
    \includegraphics[width=\linewidth]{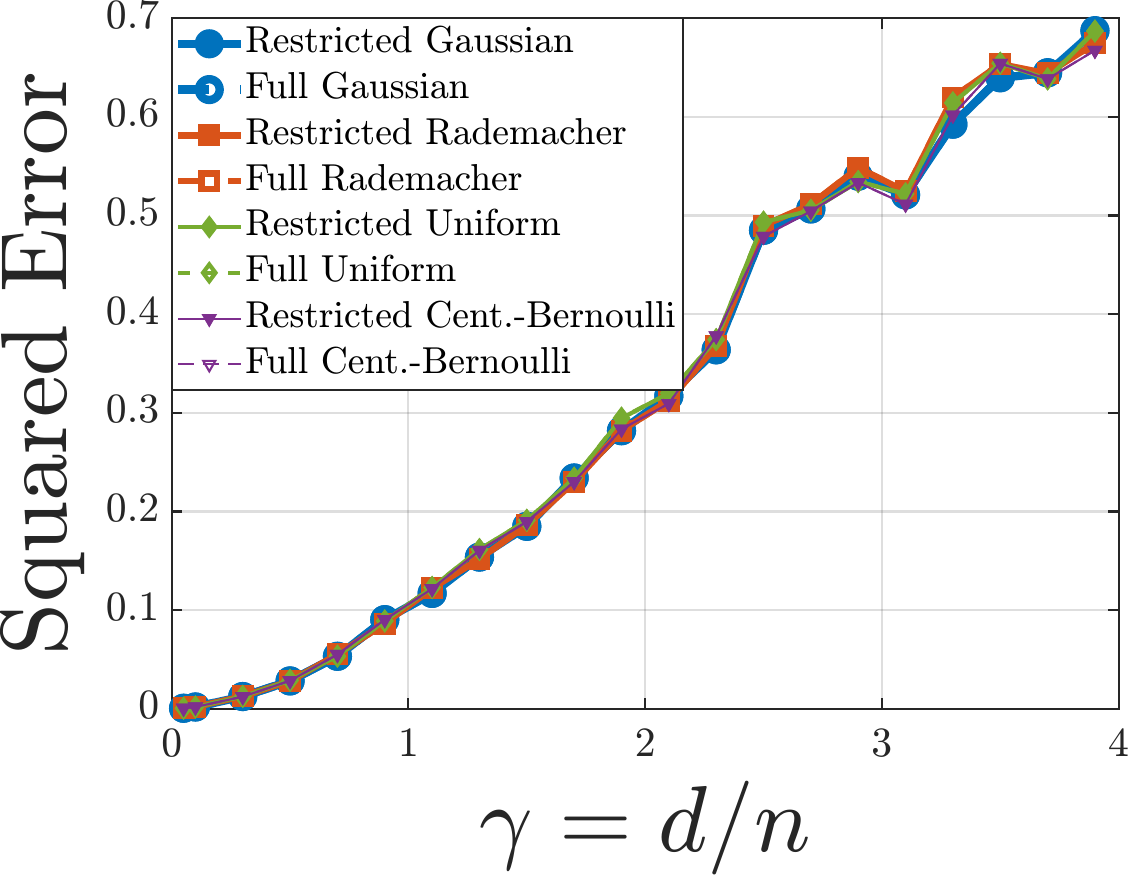}
    \caption{Anchor-diffuse profile ($r=0.25$, $\vartheta=0.75$)}
\end{subfigure}
\caption{SE curves for DST-dependent matrices under the three signal profiles. Each panel compares the oracle and full lasso for Gaussian, Rademacher, uniform, and centered-Bernoulli latent matrices, with each curve averaged over $50$ trials.}
\label{fig:sim-dst}
\end{figure}

\begin{figure*}[!htb]
\centering
\begin{subfigure}[t]{0.31\linewidth}
    \includegraphics[width=\linewidth]{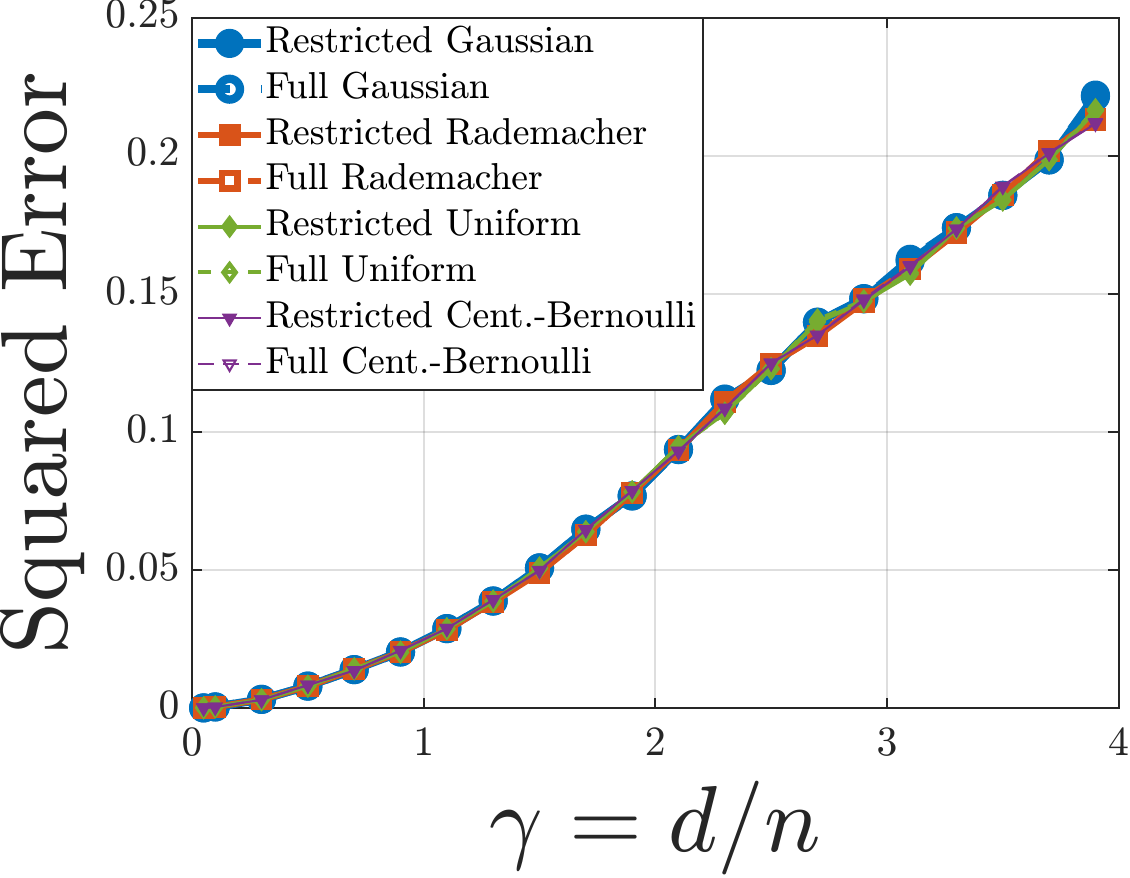}
    \caption{Geometric profile ($r=0.25$)}
\end{subfigure}
\hfill
\begin{subfigure}[t]{0.31\linewidth}
    \includegraphics[width=\linewidth]{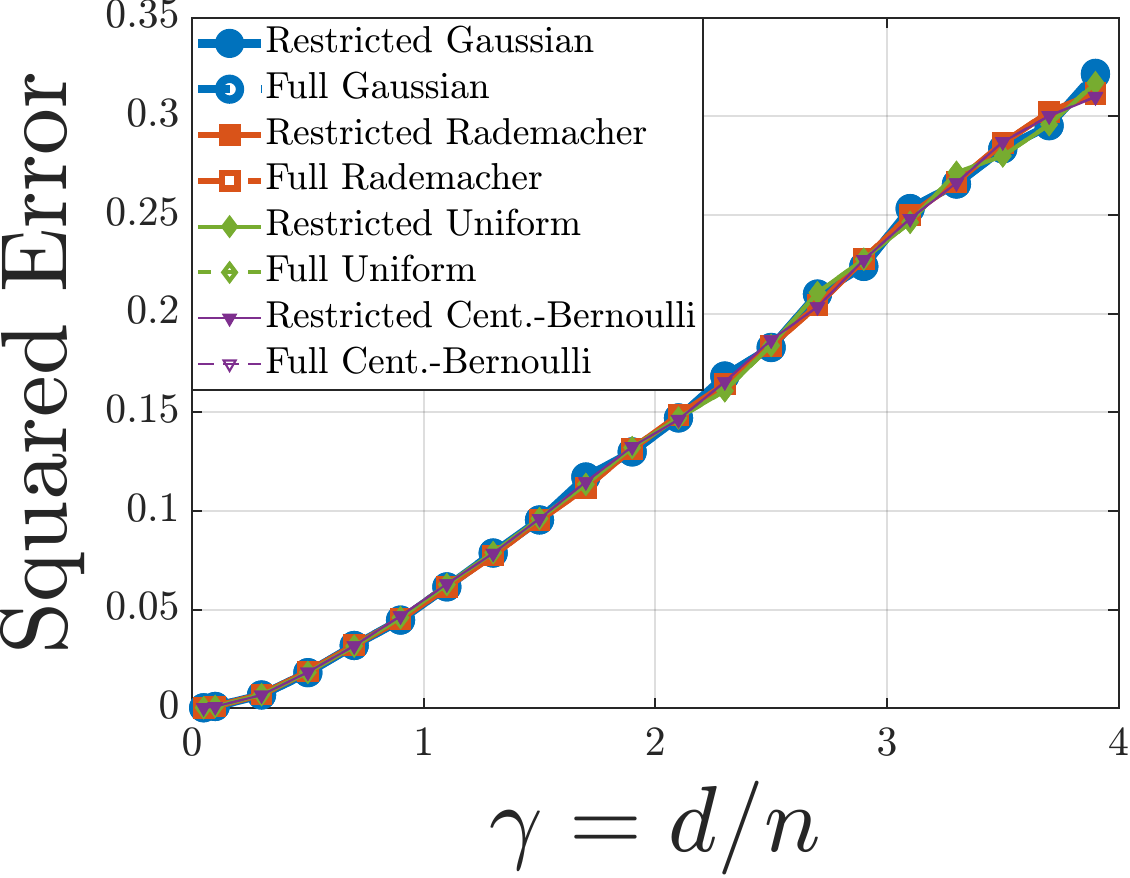}
    \caption{Power profile ($p=1.5$)}
\end{subfigure}
\hfill
\begin{subfigure}[t]{0.31\linewidth}
    \includegraphics[width=\linewidth]{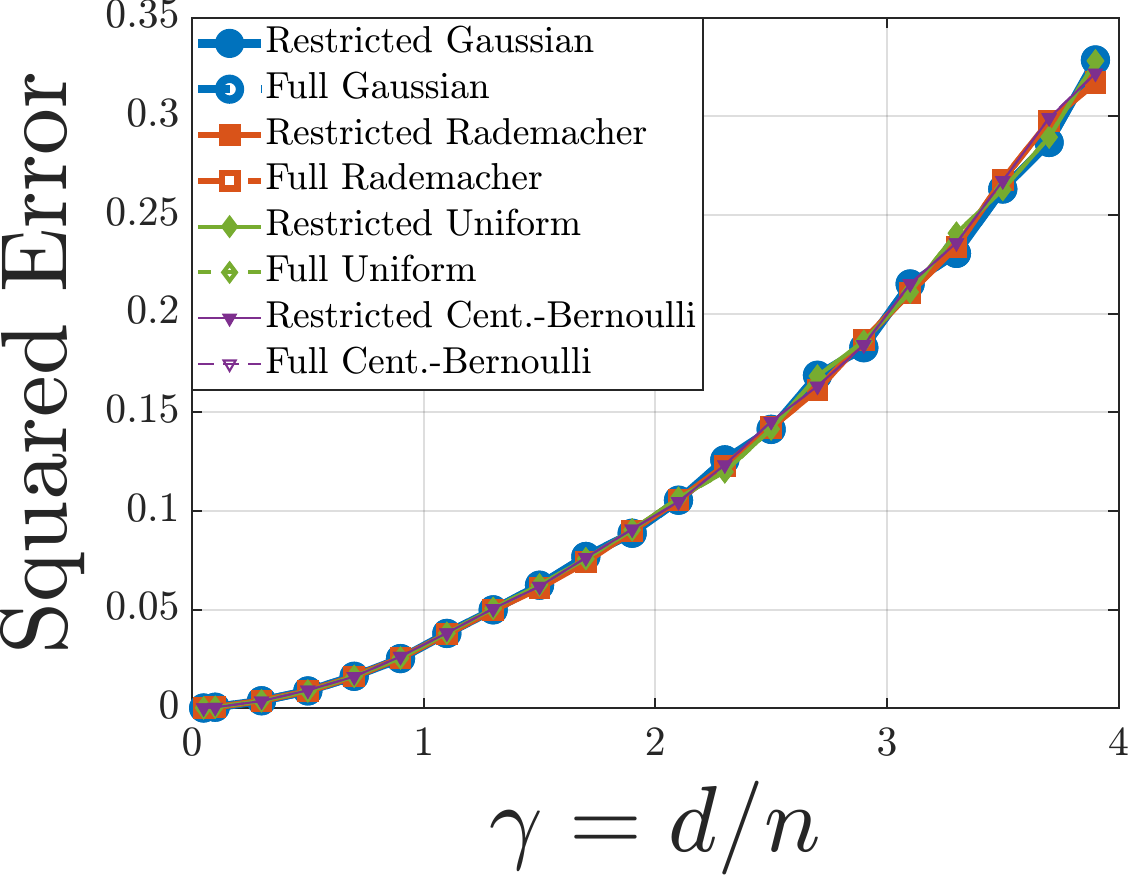}
    \caption{Anchor-diffuse profile ($r=0.25$, $\vartheta=0.75$)}
\end{subfigure}
\caption{SE curves for uniformly random orthogonal matrices under the three signal profiles. The simulation protocol and latent matrix distributions are the same as in Fig.~\ref{fig:sim-dst}.}
\label{fig:sim-haar}
\end{figure*}

\begin{figure*}[!htb]
\centering
\begin{subfigure}[t]{0.31\linewidth}
    \includegraphics[width=\linewidth]{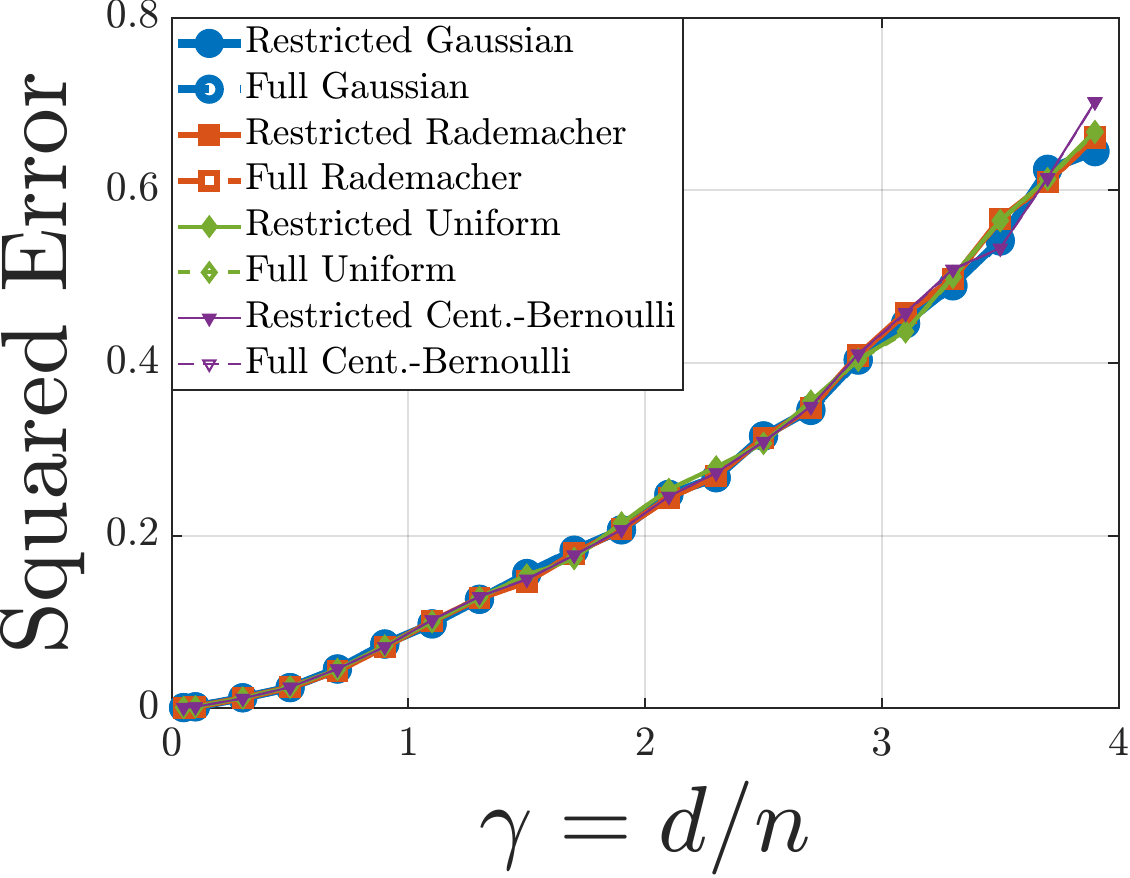}
    \caption{Geometric profile ($r=0.25$)}
\end{subfigure}
\hfill
\begin{subfigure}[t]{0.31\linewidth}
    \includegraphics[width=\linewidth]{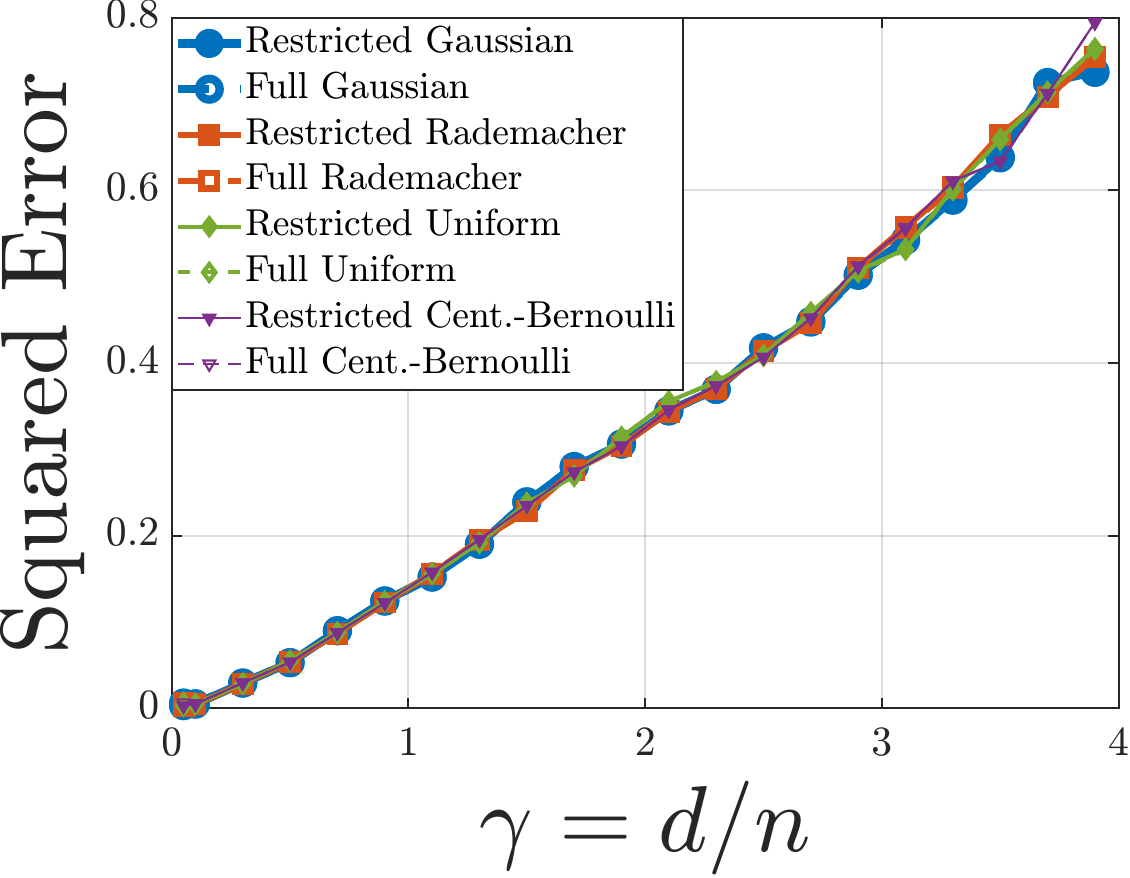}
    \caption{Power profile ($p=1.5$)}
\end{subfigure}
\hfill
\begin{subfigure}[t]{0.31\linewidth}
    \includegraphics[width=\linewidth]{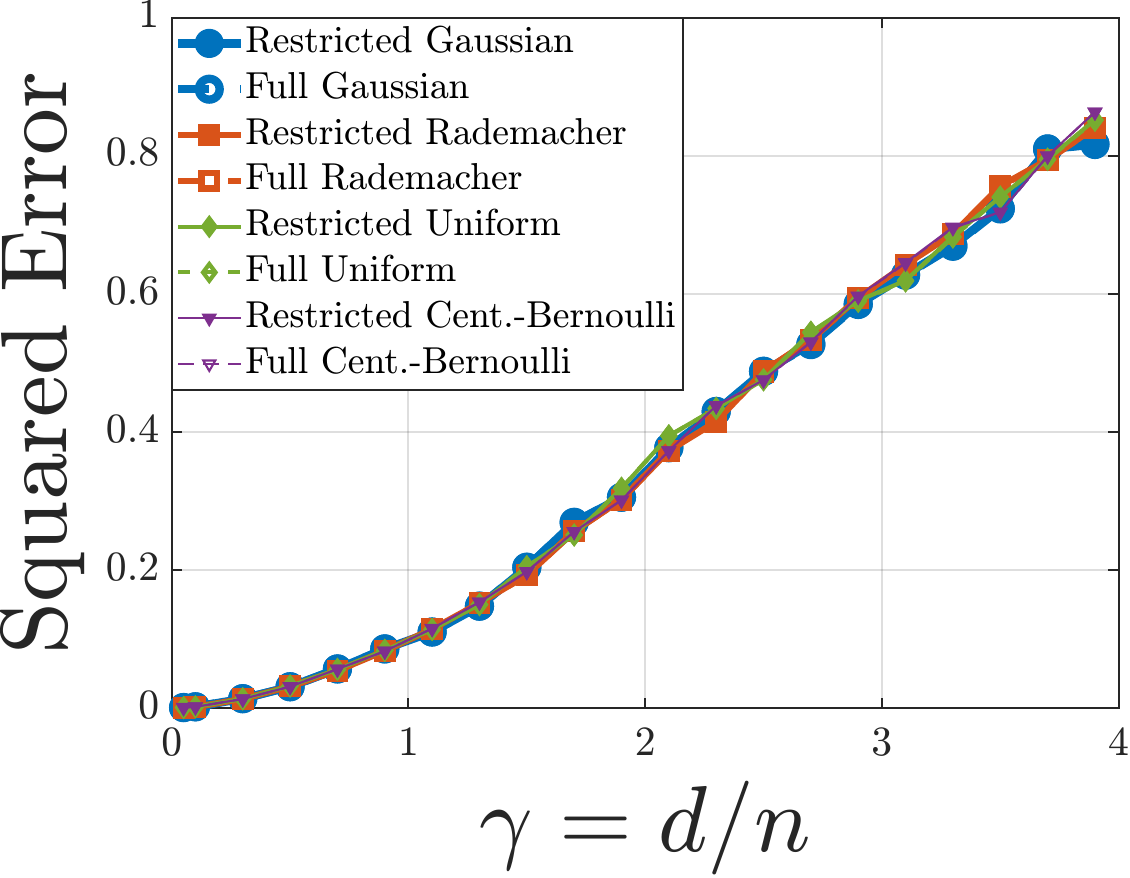}
    \caption{Anchor-diffuse profile ($r=0.25$, $\vartheta=0.75$)}
\end{subfigure}
\caption{SE curves for scaled Gaussian-dependent matrices under the three signal profiles. The simulation protocol and latent matrix distributions are the same as in Fig.~\ref{fig:sim-dst}.}
\label{fig:sim-gaussian}
\end{figure*}

In our experiments we consider the model \eqref{eq:mainproblem}, together with the $X=AZB$ covariate model. We let $n= 500$, and create a grid $\gamma \in[0.05,4]$, and let $d = \lceil \gamma n\rceil$, accordingly. We generate the sparse signal vector $\beta^\star$ with $\|\beta^\star\|_2^2 = 1$, and choose the support size to respect Assumption~\ref{ass:Sparse} as $k = \lfloor n^\frac{1}{3}\rfloor
$. We repeat the experiment for the three signal profiles in Remark~\ref{rem:remarkaboutsignals}: the geometrically decaying profile \eqref{eq:geomsig} with decay rate $r=0.25$, the normalized power profile \eqref{eq:powersig} with exponent $p=1.5$, and the anchor-diffuse profile \eqref{eq:anchorsig} with  $r=0.25$, and $\vartheta=0.75$, so that $75\%$ of the signal energy is placed on the anchor coordinate.
Furthermore, we generate the i.i.d.\ noise vector $\xi$ in \eqref{eq:mainproblem} with $\sigma=0.2$. We further note that we set the lasso regularization parameter to $\lambda = \gamma$ in all simulations and obtain the lasso solution via the fast iterative shrinkage-thresholding algorithm (FISTA) \cite{beck2009fast}.

Throughout all the simulations, for each matrix type and each value of $\gamma$, we calculate the average over $50$ trials. In each trial, we compare the SEs of the full and oracle lasso problems for different sub-Gaussian distributions of the $Z$ matrix. Namely, we compare the standard Gaussian, Rademacher, Uniform, and Bernoulli-centered distributions with a success probability of $0.2$ cases.

To effectively demonstrate the robustness of the linearly dependent framework, we choose $A$ and $B$ from three different models. Firstly, we consider the generic form of $A = U_A D_A$ and $B = U_B D_B$. In this setting, $U_A, U_B$ are DST matrices. Moreover, $D_A, D_B$ are diagonal matrices with a linear spectrum inspired by \cite{moniri_hassani_dependent_ridge,dudeja2024spectral}. Secondly, we consider both $A$ and $B$ to be drawn from a scaled Gaussian distribution as $A\sim \mathcal{N}(0,\frac{2.25I_n}{n})$ and $B\sim \mathcal{N}(0,\frac{2.25I_d}{d})$. Lastly, we consider uniformly drawn orthogonal matrices to define $A$ and $B$. That is, $A = Q_AD_AQ_A^\top $ and $B = Q_BD_BQ_B^\top $, where $Q_A$ and $Q_B$ are uniformly drawn orthogonal matrices. Furthermore,  $\mu_A = \frac{1}{3}\delta_2 + \frac{1}{3}\delta_4+\frac{1}{3}\delta_6
$ governs the spectrum of $A^\top A$, and $\mu_B = \frac{1}{4}\delta_1 + \frac{1}{4}\delta_2+\frac{1}{4}\delta_3+\frac{1}{4}\delta_4
$ governs the spectrum of $B^\top B$, respectively. These choices of $A$ and $B$ are all dense matrices that allow mixing of the coordinates, thereby truly achieving the linearly dependent model we described.



Figures~\ref{fig:sim-dst}, \ref{fig:sim-haar}, and \ref{fig:sim-gaussian} show the results for the DST, uniformly random orthogonal, and scaled Gaussian mixing matrices, respectively, with one panel for each signal profile. Across all nine settings, the SE values for the Rademacher, uniform, and centered-Bernoulli distributions closely track the Gaussian curve. Furthermore, the full and oracle estimator curves remain closely matched for all distributions and signal profiles, illustrating the oracle-to-full universality transfer mechanism predicted by our theoretical results.



\section{Conclusion}\label{Sec:SectionVI}

We established Gaussian universality of the lasso for the linearly dependent design $X= AZB$ in the sparse regime $k= o(n)$. The oracle estimator concentrates around a deterministic lasso problem, and a directional strict dual-feasibility condition makes this estimator the unique full lasso solution with high probability. We verified the conditions and nontriviality thresholds for example deterministic and random mixing matrix families. The simulations illustrate the resulting universality across several latent distributions and sparse signal profiles.

\section*{Acknowledgments}
The authors used ChatGPT (OpenAI) for assistance with exposition, organization, notation, and proof refinement during the preparation of this manuscript. The authors conceived the project, developed the mathematical framework, proved the results, and made all final editorial and mathematical decisions.

\section*{Funding}
This work was supported in part by a Hellman Fellowship from the University of California.

\bibliographystyle{abbrv}  
\bibliography{refs_final_Consistent}

\appendix
\renewcommand*{\theHsection}{appendix.\Alph{section}}
\renewcommand*{\theHsubsection}{appendix.\Alph{section}.\arabic{subsection}}
\renewcommand*{\theHequation}{appendix.\Alph{section}.\arabic{equation}}

\section{Proof of Proposition~\ref{prop:conditional-transfer}}\label{app:PropTransProof}

\begin{proof}
We consider all of the coefficients and definitions provided throughout the paper, especially Section~\ref{Sec:SectionII}, and therefore do not repeat them here. As stated in the proposition, we assume $\mathbb{P}(\mathcal{G}_n)\to 1$. We first prove that the convergence-in-probability results in Proposition~\ref{prop:RestrictedProp} and Theorem~\ref{thm:lassoFullUniversality} hold. On $\mathcal{F}_n$, the tuple $\Theta_n$ is independent of the latent matrix, its Gaussian counterpart, and the noise vector. Let $\theta\in\mathcal{C}_n$ be an admissible deterministic tuple, and let $\mathbb{P}_\theta$ denote the probability where only $(W,\xi)$ are random. Fix $W\in\{Z,G\}$. We claim that, for every $\epsilon>0$,
\begin{equation}\label{eq:whatweprove1}
r_{n}^\mathrm{oracle}(\epsilon) \coloneqq \sup_{\theta\in \mathcal{C}_n} \mathbb P_\theta\left(
\big\|\hat\beta_S^{\mathrm{oracle}}(W)-\hat{\beta}^{\mathrm{det}}\big\|_2>\epsilon
\right)\to 0.
\end{equation}
We prove this by contradiction. Suppose $r_{n}^\mathrm{oracle}(\epsilon)\not\to0$. Then, there exist a subsequence $(n_m)_m$ and a constant $c_0>0$ such that $r_{n_m}^\mathrm{oracle}(\epsilon)\geq c_0$. By the definition of the supremum, there exists $\theta_{n_m}\in \mathcal{C}_{n_m}$ such that
\begin{equation}\label{eq:contradiction}
\mathbb{P}_{\theta_{n_m}}\left(
\big\|\hat\beta_S^{\mathrm{oracle}}(W)-\hat{\beta}^{\mathrm{det}}\big\|_2>\epsilon
\right)\geq \frac {c_0} 2.
\end{equation}
The triangular array $(\theta_{n_m})_m$ satisfies Assumptions~\ref{ass:Sparse}, \ref{ass:covstruct}, and~\ref{ass:lassoUniqueness} with the same constants and with $|S|\leq\bar k_{n_m}=o(n_m)$. Proposition~\ref{prop:RestrictedProp}, applied along this subsequence, makes the probability in \eqref{eq:contradiction} tend to zero, which is a contradiction and proves \eqref{eq:whatweprove1}.

Next, define the failure event
\begin{equation}
\mathcal{E}_{n,\epsilon}^\mathrm{oracle} \coloneqq \left\{\big\|\hat\beta_S^{\mathrm{oracle}}(W)-\hat{\beta}^{\mathrm{det}}\big\|_2>\epsilon\right\}.
\end{equation}
On $\mathcal G_n$, we have $\mathbb{P}(\mathcal{E}_{n,\epsilon}^\mathrm{oracle}\mid \mathcal{F}_n)\leq r_{n}^\mathrm{oracle}(\epsilon)$. Therefore, by the tower property,
\begin{align}
\mathbb{P}\left(\mathcal{E}_{n,\epsilon}^\mathrm{oracle}\right)
&=\mathbb E\left[\mathbf1_{\mathcal G_n}\mathbb P(\mathcal{E}_{n,\epsilon}^\mathrm{oracle}\mid\mathcal F_n)\right]
+\mathbb E\left[\mathbf1_{\mathcal G_n^c}\mathbb P(\mathcal{E}_{n,\epsilon}^\mathrm{oracle}\mid\mathcal F_n)\right]\nonumber\\
&\leq r_{n}^\mathrm{oracle}(\epsilon)+\mathbb P(\mathcal G_n^c)\to 0.
\end{align}
This proves the conclusion of Proposition~\ref{prop:RestrictedProp}. The same contradiction and conditioning argument, applied to the failure event that the full lasso is not unique or does not equal the oracle estimator, proves the conclusion of Theorem~\ref{thm:lassoFullUniversality}.

Next, we justify the expectation-level statements in Corollaries~\ref{cor:Corollary} and~\ref{cor:Corollary2}. Lemma~\ref{lem:moment} is the key ingredient in the deterministic case. Conditioned on $\Theta_n$, its proof gives
\begin{equation}
\mathbb E\left[\|AWB\beta^\star\|_2^8\mid\Theta_n\right]
\leq Cn^4\alpha^8\|A\|_{\mathrm{op}}^8\|B\|_{\mathrm{op}}^8.
\end{equation}
Taking expectations and applying the Cauchy--Schwarz inequality, under condition \eqref{eq:preliminary-moment-condition} we have
\begin{align}
\mathbb E\|AWB\beta^\star\|_2^8
&\leq Cn^4\alpha^8
\mathbb E[\|A\|_{\mathrm{op}}^8\|B\|_{\mathrm{op}}^8]\nonumber\\
&\leq Cn^4\alpha^8
(\mathbb E\|A\|_{\mathrm{op}}^{16})^{1/2}
(\mathbb E\|B\|_{\mathrm{op}}^{16})^{1/2}\leq C'n^4.
\end{align}
Furthermore, since $\mathbb{E}\|\xi\|_2^8=O(n^4)$, Lemma~\ref{lem:moment} holds. The uniform-integrability arguments used in Corollaries~\ref{cor:Corollary} and~\ref{cor:Corollary2} therefore apply unchanged.
\end{proof}

\section{Proof of \eqref{eq:eqprovegauss}}\label{app:appB}

\begin{proof}
Let $\mathcal H_n\coloneqq\sigma(A,B_S,S,\beta^\star)$. Because $(S,\beta^\star)$ is independent of $B$, each inactive column $b_j$, $j\in S^c$, is independent of $\mathcal H_n$ and distributed as $\mathcal N(0,(\omega/d)I_d)$. Since $h_n$ is $\mathcal H_n$-measurable,
\begin{equation}
\tau_nb_j^\top B_Sh_n\mid\mathcal H_n
\sim\mathcal N\left(0,\frac{\omega\tau_n^2}{d}\|B_Sh_n\|_2^2\right).
\end{equation}
Choose constants $C_\tau,C_B<\infty$ such that
\begin{equation}
\mathcal E_n\coloneqq\{\tau_n\leq C_\tau,\ \|B_S\|_\mathrm{op}\leq C_B\}
\end{equation}
has probability tending to one. On $\mathcal E_n$, optimality of $\hat\beta^\mathrm{det}$ gives
\begin{equation}
\|\hat\beta^\mathrm{det}\|_1
\leq\frac{\tau_n\|B_S\beta_S^\star\|_2^2}{2\lambda}
\leq\frac{C_\tau C_B^2\alpha^2}{2\lambda},
\end{equation}
and therefore $\|B_Sh_n\|_2\leq H$ for a deterministic constant $H<\infty$. Thus, on $\mathcal E_n$, the conditional variance above is at most $V/d$, where $V\coloneqq\omega C_\tau^2H^2$. The Gaussian tail bound and a union bound yield
\begin{equation}
\mathbb P\left(\max_{j\in S^c}|\tau_nb_j^\top B_Sh_n|>t\right)
\leq\mathbb P(\mathcal E_n^c)+2d\exp\left(-\frac{dt^2}{2V}\right).
\end{equation}
Taking $t=M\sqrt{\log(d)/d}$ with $M^2>2V$ sufficiently large proves \eqref{eq:eqprovegauss}.
\end{proof}

\section{Proofs of \eqref{eq:haar-active-block} and \eqref{eq:proveHaarGood}}\label{app:appC}

In this section, we condition on $S$ throughout. Let $C=B^\top B=Q_B\Lambda Q_B^\top$, where $\Lambda \coloneqq D_B^2$, $\|\Lambda\|_\mathrm{op}\leq L_B$, and $\bar b_d=d^{-1}\operatorname{tr}(\Lambda)$. Consider $P_S\in \mathbb{R}^{d\times k}$ such that $B_S=BP_S$.

\begin{proof}[Proof of \eqref{eq:haar-active-block}]
For fixed $u\in\mathbb S^{k-1}$, define
\begin{equation}
f_u(Q)\coloneqq u^\top P_S^\top Q\Lambda Q^\top P_Su.
\end{equation}
The function $f_u$ is $2L_B$-Lipschitz under the Frobenius metric, and invariance of the Haar measure gives $\mathbb Ef_u(Q_B)=\bar b_d$. The concentration inequality on the special orthogonal group~\cite[Theorem~5.2.7]{vershynin2018hdp} therefore applies on each connected component of the orthogonal group. Indeed, right multiplication by a diagonal reflection that commutes with $\Lambda$ maps the two components onto one another without changing $f_u$. Hence,
\begin{equation}
\mathbb P\bigl(|f_u(Q_B)-\bar b_d|>t\bigr)
\leq2\exp\left(-\frac{c_0dt^2}{L_B^2}\right).
\end{equation}
Let $\mathcal N$ be a $1/4$-net of $\mathbb S^{k-1}$ with $|\mathcal N|\leq9^k$. The standard net bound for symmetric matrices and a union bound imply
\begin{equation}
\mathbb P\left(\|P_S^\top CP_S-\bar b_dI_k\|_\mathrm{op}>2t\right)
\leq2\exp\left(k\log9-\frac{c_0dt^2}{L_B^2}\right).
\end{equation}
For each fixed $t>0$, the right-hand side tends to zero because $k=o(d)$. Since $\bar b_d\to\bar b$, this proves \eqref{eq:haar-active-block}.
\end{proof}

\begin{proof}[Proof of \eqref{eq:proveHaarGood}]
It suffices to prove
\begin{equation}\label{eq:eq98s}
\max_{j\in S^c} \|r_j\|_2 = O_\mathbb{P}\left(\sqrt{\frac{k+\log(d)}{d}}\right),
\end{equation}
where $r_j=P_S^\top Ce_j$. For a fixed $j\in S^c$ and $u\in\mathbb{S}^{k-1} \coloneqq \{x\in\mathbb{R}^k \mid \|x\|_2 =1\}$, define $g_{j,u}(Q) \coloneqq e_j^\top Q\Lambda Q^\top P_Su$. This function is $2L_B$-Lipschitz under the Frobenius norm metric. Indeed, for $Q,Q'\in O(d)$,
\begin{align}
\left|g_{j,u}(Q)-g_{j,u}(Q')\right|&\leq \left|e_j^\top (Q-Q')\Lambda Q^\top P_S u\right|+
\left|e_j^\top Q'\Lambda (Q-Q')^\top P_S u\right|\leq 2L_B\|Q-Q'\|_F.
\end{align}
Moreover, $\mathbb Eg_{j,u}(Q_B)=\bar b_de_j^\top P_Su=0$. The same concentration inequality therefore gives
\begin{equation}
\mathbb{P}\left(|g_{j,u}(Q_B)|>t\right)\leq2\mathrm{e}^{-\frac{c_0dt^2}{L_B^2}}.
\end{equation}
Next, let $\mathcal{N}$ be a $\frac{1}{2}$-net of $\mathbb{S}^{k-1}$ with $|\mathcal N|\leq5^k$. Then, $\|r_j\|_2\leq 2\max_{v\in \mathcal{N}}|r_j^\top v|$. By a union bound over $j\in S^c$ and $v\in \mathcal{N}$, we have
\begin{align}
\mathbb{P}\left(\max_{j\in S^c}\|r_j\|_2>2t\right)&\leq\sum_{j\in S^c}\sum_{v\in \mathcal{N}}\mathbb{P}\left(|r_j^\top v|>t\right)
\leq 2\exp\left[\log(d)+k\log(5)-\frac{c_0dt^2}{L_B^2}\right].
\end{align}
Choosing $t=M\sqrt{\frac{k+\log (d)}{d}}$, with $c_0M^2/L_B^2>\max\{1,\log5\}$, proves \eqref{eq:eq98s}. Finally,
\begin{equation}
\max_{j\in S^c}|\tau_nb_j^\top B_Sh_n|
\leq\tau_n\|h_n\|_2\max_{j\in S^c}\|r_j\|_2.
\end{equation}
Bounded spectra imply $\sup_n\tau_n<\infty$ and optimality of $\hat\beta^\mathrm{det}$ gives $\sup_n\|h_n\|_2<\infty$. This proves \eqref{eq:proveHaarGood}.
\end{proof}

\section{Proof of Lemma~\ref{lem:moment}}\label{app:appD}

\begin{proof}
Let $v =B\beta^\star$  since $A$ and $B$ are bounded in operator norm, we denote $\|A\|_\mathrm{op}\leq C_A$, and $\|B\|_\mathrm{op}\leq C_B$, and observe that $v$ is deterministic and bounded
$\|v\|_2 = \|B\beta^\star\|_2\leq\|B\|_\mathrm{op}\|\beta^\star\|_2=C_B\alpha$.

In the next step, we consider $Wv$ that is made from independent, sub-Gaussian random variables with uniformly bounded $8$-th moments. To show this mathematically, let 
\begin{equation}
Y_i  \coloneqq  (Wv)_i
= \sum_{j=1}^d W_{ij}v_j,
\quad i=1,\ldots,n,
\end{equation}
therefore $\|Wv\|_2^2=\sum_{i=1}^n Y_i^2$. By Khintchine's inequality~\cite[Theorem~2.7.5]{vershynin2018hdp}
for $p=8$, we have
\begin{align}
\|Y_i\|_{L^8} = \left\|\sum_{j=1}^dv_j W_{ij}\right\|_{L^8}\leq 2\sqrt{2}C_1K\left(\sum_{j=1}^d v_j^2 \right) ^\frac{1}{2}
\leq 2\sqrt{2}C_1KC_B\alpha \eqqcolon C.
\end{align}
Consequently, $\sup_{n,i}\mathbb{E}|Y_i|^8\leq C$, which confirms the boundedness claim. In the next step, we want to show $\mathbb{E}\|Wv\|_2^8\leq C_2n^4$. Let

\begin{equation}
\mathbb{E}\|Wv\|_2^8
= \mathbb{E}\left(\sum_{i=1}^n Y_i^2\right)^4 = \sum_{i_1,i_2,i_3,i_4=1}^n \mathbb{E}[Y_{i_1}^2Y_{i_2}^2Y_{i_3}^2Y_{i_4}^2],
\end{equation}
By the general form of Hölder's inequality (see~\cite[Theorem~1]{mitrinovic2013classical}), we have
\begin{align}\label{eq:Holder}
\mathbb{E}\left[Y_{i_1}^2Y_{i_2}^2Y_{i_3}^2Y_{i_4}^2\right]&\leq \mathbb{E}\left[|Y_{i_1}^2|\cdot |Y_{i_2}^2| \cdot |Y_{i_3}^2|\cdot |Y_{i_4}^2|\right]\leq \prod_{r=1}^4\left(\mathbb{E}|Y_{i_r}|^8\right)^\frac 1 4
\end{align}
Considering the $8$-th moment bound and \eqref{eq:Holder} we have $\mathbb{E}\|Wv\|_2^8\leq \sum_{i_1,i_2,i_3,i_4 = 1}^n C_2 = C_2n^4$. Next, we consider the noise term and want to show $\mathbb{E}\|\xi\|_2^8\leq C_3n^4$. Let
\begin{equation}
\mathbb{E}\|\xi\|_2^8 = \mathbb{E}\left(\sum_{i=1}^n[\xi_i^2]\right)^4 = \sum_{i_1,i_2,i_3,i_4=1}^n\mathbb{E}\left[\xi_{i_1}^2\xi_{i_2}^2\xi_{i_3}^2\xi_{i_4}^2\right]. 
\end{equation}
Similar to \eqref{eq:Holder} we apply Hölder's inequality and obtain
\begin{align}\label{eq:Holder2}
\mathbb{E}\left[\xi_{i_1}^2\xi_{i_2}^2\xi_{i_3}^2\xi_{i_4}^2\right]&\leq \mathbb{E}\left[|\xi_{i_1}^2| \cdot |\xi_{i_2}^2|\cdot |\xi_{i_3}^2|\cdot |\xi_{i_4}^2|\right]\leq \prod_{r=1}^4\left(\mathbb{E}|\xi_{i_r}|^8\right)^\frac 1 4.
\end{align}
To proceed, we must obtain the $8$-th moment of each noise coordinate $\xi_i \sim \mathcal{N}(0,\sigma^2)$. A well-known identity of zero-mean Gaussian random variables such as $\xi_i$ is $\mathbb{E}[\xi_i^{2m}] = \sigma^{2m}(2m-1)!!$ (double factorial), therefore
\begin{align}\label{eq:noisebound}
\mathbb{E}\left[\xi_{i_1}^2\xi_{i_2}^2\xi_{i_3}^2\xi_{i_4}^2\right]&\leq \mathbb{E}\left[|\xi_{i_1}^2|\cdot |\xi_{i_2}^2|\cdot |\xi_{i_3}^2|\cdot |\xi_{i_4}^2|\right]\leq 105\sigma^8\leq C_3,\nonumber\\
\mathbb{E}\|\xi\|_2^8&\leq \sum_{i_1,i_2,i_3,i_4=1}^n C_3=C_3n^4. 
\end{align}
Next, we prove the bound $\mathbb{E}\|y(W)\|_2^8\leq C_4n^4$. By the triangle inequality, we have $\|y(W)\|_2 = \|AWv+\xi\|_2 \leq \|AWv\|_2+\|\xi\|_2$,
Furthermore,
$\mathbb{E}\|y(W)\|_2^8
\leq 2^7\mathbb{E}\|AWv\|_2^8
+2^7\mathbb{E}\|\xi\|_2^8$.
By standard norm bounds and Assumption~\ref{ass:covstruct} we have $\|AWv\|_2\leq \|A\|_\mathrm{op}\|Wv\|_2\leq C_A\|Wv\|_2$, which leads to
\begin{equation}\label{eq:yfirstgoomb}
\begin{aligned}
\mathbb{E}\|AWv\|_2^8\leq C_A^8\mathbb{E}\|Wv\|_2^8\leq C_A^8C_2n^4 \coloneqq Q_1n^4,
\end{aligned}
\end{equation}
Considering \eqref{eq:noisebound} and \eqref{eq:yfirstgoomb}, we deduce that
\begin{equation}
\begin{aligned}
\mathbb{E}\|y(W)\|_2^8 \leq 2^7(Q_1+C_3)n^4 \coloneqq C_4n^4,
\end{aligned}
\end{equation}
Therefore, we have the following bound
\begin{equation}\label{eq:firstboundproof}
\sup_n\mathbb{E}\left(\frac{\|y(W)\|_2^2}{n}\right)^4<\infty.
\end{equation}

Let $\hat{\beta}$ denote either the full estimator in \eqref{eq:lassoproblem} or the oracle estimator in \eqref{eq:OracleObj}. The zero vector is feasible for both \eqref{eq:lassoproblem} and \eqref{eq:OracleObj}. This is apparent from \eqref{eq:lassoproblem}, and feasibility holds for \eqref{eq:OracleObj} because $\mathrm{supp}(0) = \emptyset \subseteq S$; therefore, feasibility holds. Therefore, for both objective functions and their supports, we have
\begin{align}\label{eq:optbetaineq}
\frac{1}{2n}\left\|y(W)-X(W)\hat{\beta}\right\|_2^2
+\lambda\|\hat{\beta}\|_1
\leq
\frac{1}{2n}\left\|y(W)-X(W)\cdot 0\right\|_2^2
=\frac{1}{2n}\|y(W)\|_2^2.
\end{align}
By ignoring the nonnegative terms on the left-hand side of \eqref{eq:optbetaineq}, we have
\begin{equation}\label{eq:optimalbound}
\begin{aligned}
\lambda\|\hat{\beta}\|_1
&\leq\frac{1}{2n}\|y(W)\|_2^2.
\end{aligned}
\end{equation}
To bound $\|\hat{\beta}-\beta^\star\|_2^2$, consider
$\|\hat{\beta}-\beta^\star
\|_2^2\leq
 2\|\hat{\beta}\|_2^2+2\|\beta^\star\|_2^2\leq 2\|\hat{\beta}\|_2^2+2\alpha^2$.
Since $\|\hat{\beta}\|_2\leq \|\hat{\beta}\|_1$, we deduce that
\begin{equation}
    \|\hat{\beta}-\beta^\star\|_2^2\leq \frac{\|y(W)\|_2^4}{2n^2\lambda^2}+2\alpha^2.
\end{equation}
Furthermore, we have
\begin{equation}
\begin{aligned}
 \|\hat{\beta}-\beta^\star\|_2^4\leq \frac{\|y(W)\|_2^8}{2n^4\lambda^4}+8\alpha^4
\end{aligned}
\end{equation}
Taking expectations provides
\begin{equation}\label{eq:boundbeforefinal}
\mathbb{E}\|\hat{\beta}-\beta^\star\|_2^4\leq \frac{1}{2\lambda^4}\mathbb{E}\left(\frac{\|y(W)\|_2^2}{n}\right)^4+8\alpha^4.
\end{equation}
By \eqref{eq:bound1} and \eqref{eq:boundbeforefinal}, taking the supremum yields \eqref{eq:bound2}, concluding the proof.
\end{proof}

\section{Proof of Lemma~\ref{lem:bilinear}}\label{app:appE}

\begin{proof}

Let $H = A^\top A\in\mathbb{R}^{n\times n}$. To begin the proof, consider the following preliminaries 
\begin{align}
\|H\|_\mathrm{op} &= \|A^\top A\|_\mathrm{op} = \|A\|_\mathrm{op}^2\leq C_A^2,\nonumber\\
\|H\|_F &= \|A^\top A\|_F\leq\|A^\top\|_\mathrm{op}\|A\|_F = \|A\|_\mathrm{op}\|A\|_F,
\end{align}
Furthermore, since $\|A\|_F^2 = \sum_{r=1}^ns_r(A)^2\leq n\|A\|_\mathrm{op}^2$, where $s_r(A)$ is the $r$-th singular value of $A$, we have $\|H\|_F\leq C_A^2\sqrt{n}$.
Let $x \coloneqq \mathrm{vec}(W)\in\mathbb{R}^{nd}$, and let $M \coloneqq (pq^\top)\otimes H\in\mathbb{R}^{nd\times nd}$, and consider
\begin{align}
p^\top W^\top HWq
&= \sum_{i=1}^d\sum_{j=1}^d p_iq_j(W^\top HW)_{ij}
= \sum_{i,j=1}^d\sum_{k,\ell=1}^n
p_iq_jH_{k\ell}W_{ki}W_{\ell j}.
\end{align}
We further observe,
\begin{align}
x^\top Mx
&=\sum_{i,j=1}^d\sum_{k,\ell=1}^n
(pq^\top)_{ij}H_{k\ell}W_{ki}W_{\ell j}=\sum_{i,j=1}^d\sum_{k,\ell=1}^n
p_iq_jH_{k\ell}W_{ki}W_{\ell j}.
\end{align}
So, $x^\top Mx=p^\top W^\top HWq$. Let $M_\mathrm{Sym} = \frac
{M+M^\top}{2}$, therefore, $x^\top Mx = x^\top M_\mathrm{Sym}x$ for all real $x$. Consider,
\begin{equation}
\|M_\mathrm{Sym}\|_\mathrm{F}
=\left\|\frac{M+M^\top}{2}\right\|_\mathrm{F}\leq \frac{1}{2}\|M\|_\mathrm{F}
+\frac{1}{2}\|M^\top\|_\mathrm{F}=\|M\|_\mathrm{F},
\end{equation}
where $\|M\|_\mathrm{F}
=\|(pq^\top)\otimes H\|_\mathrm{F}
=\|p\|_2\|q\|_2\|H\|_F$. This gives $\|M_\mathrm{Sym}\|_F
\leq L^2C_A^2\sqrt{n}$. Similarly, $\|M_\mathrm{Sym}\|_\mathrm{op}
\leq L^2C_A^2$ We know that $x$ has independent, zero-mean, and sub-Gaussian entries. Furthermore, $R = \frac{1}{n}M_\mathrm{Sym}$ is deterministic; thus, the Hanson-Wright inequality ( \cite[Theorem~6.2.2]{vershynin2018hdp}) gives
\begin{equation}
\mathbb{P}\left(\left|x^\top Rx-\mathbb{E}[x^\top Rx]\right|>t\right)
\leq 2\exp\left\{-c\min\left(
\frac{t^2}{K^4\|R\|_F^2},
\frac{t}{K^2\|R\|_\mathrm{op}}
\right)\right\},
\end{equation}
where $K \coloneqq \sup_{n,i,j}\|W_{ij}\|_{\psi_2}$. By definition of $R$, we deduce that $\|R\|_\mathrm{op}\leq \frac{L^2C_A^2}{n}$ and $\|R\|_F \leq
 \frac{L^2C_A^2}{\sqrt{n}}$. Therefore, it is evident that $\frac{t^2}{K^4\|R\|_F^2}\geq \frac{nt^2}{K^4L^4C_A^4}$, and $\frac{t}{K^2\|R\|_\mathrm{op}}\geq \frac{nt}{K^2L^2C_A^2}$. Thus, given a constant $c>0$, we have
\begin{equation}\label{eq:hansonthelord}
\mathbb{P}\left(
\left|\frac{1}{n}p^\top W^\top HWq
-\mathbb{E}\left[\frac{1}{n}p^\top W^\top HWq\right]\right|>t
\right)\leq C\exp\{-cn\min(t,t^2)\},
\end{equation}
where $C\geq 2$.

To calculate $\mathbb{E}\left[\frac{1}{n}p^\top W^\top HWq\right]$, consider
\begin{align}\label{eq:part1lem2last}
\mathbb{E}\left[\frac{1}{n}p^\top W^\top HWq\right]
&= \frac{1}{n}\sum_{i,j=1}^d\sum_{k,\ell=1}^n
p_iq_jH_{k\ell}\mathbb{E}[W_{ki}W_{\ell j}]= \frac{1}{n}\sum_{i=1}^d\sum_{k=1}^n p_iq_iH_{kk}= \frac{1}{n}\left(\sum_{k=1}^nH_{kk}\right)
\left(\sum_{i=1}^d p_iq_i\right)\nonumber\\
&=\frac{1}{n}\mathrm{tr}(H)p^\top q=\tau_np^\top q.
\end{align}
By substituting \eqref{eq:hansonthelord}, in \eqref{eq:part1lem2last} we obtain \eqref{eq:probinequal1}.
 The next step is to prove \eqref{eq:probinequal2}. Let $a \coloneqq A^\top\xi\in\mathbb{R}^n$, then
\begin{align}
\frac{1}{n}p^\top W^\top A^\top\xi
&=\frac{1}{n}p^\top W^\top a
=\sum_{\ell=1}^d\sum_{i=1}^n
\frac{a_ip_\ell}{n}W_{i\ell}=\sum_{\ell=1}^d\sum_{i=1}^n c_{i\ell}W_{i\ell},
\end{align}
where $c_{i\ell} = \frac{a_ip_\ell}{n} = \frac{(A^\top\xi)_ip_\ell}{n}$ are deterministic conditioned on $\xi$. We consider the vector of coefficients as $c = (c_{il})\in\mathbb{R}^{nd}$, and obtain an upper bound on $\|c\|_2^2$, as
\begin{align}
\|c\|_2^2
=\sum_{i=1}^n\sum_{\ell=1}^d
\left(\frac{(A^\top \xi)_i p_\ell}{n}\right)^2
=\frac{1}{n^2}\left(\sum_{i=1}^n(A^\top \xi)_i^2\right)
\left(\sum_{\ell=1}^d p_\ell^2\right)
=\frac{1}{n^2}\|A^\top\xi\|_2^2\|p\|_2^2\leq \frac{L^2C_A^2\|\xi\|_2^2}{n^2}.
\end{align}

By \cite[Proposition~2.7.1]{vershynin2018hdp},
conditioned on $\xi$, $Y  \coloneqq \frac{1}{n}p^\top W^\top A^\top \xi = \sum_{i,\ell}c_{i\ell}W_{i\ell}$ is sub-Gaussian and we deduce that 
\begin{equation}
\|Y \mid \xi\|_{\psi_2}\leq 
QK\|c\|_2\leq QK\frac{LC_A\|\xi\|_2}{n}=C\frac{\|\xi\|_2}{n},
\end{equation}
where $C = QKLC_A$. By \cite[Proposition~2.6.6]{vershynin2018hdp}, we observe
\begin{align}\label{eq:Finalendoflemma}
\mathbb{P}(|Y|>t\mid \xi)
\leq 2\exp\left(-\frac{c_1t^2}{\|Y\|_{\psi_2}^2}\right)
\leq 2\exp\left(-\frac{c_1n^2t^2}{C^2\|\xi\|_2^2}\right)
=2\exp\left(-\frac{cn^2t^2}{\|\xi\|_2^2}\right),
\end{align}
where $c=\frac{c_1}{C^2}$. Finally, by rearranging \eqref{eq:Finalendoflemma} we obtain \eqref{eq:probinequal2} and conclude the proof.
\end{proof}

\section{Proof of Lemma~\ref{lem:active-conc}}\label{app:appF}

\begin{proof}
Recall from Section~\ref{subsec:Gngn} that $G_n(W) = \frac1n X_S(W)^\top X_S(W)-\tau_nB_S^\top B_S$ and is symmetric. Let $u\in\mathbb{R}^k$, such that $\|u\|_2 =1$, consider
\begin{align}
u^\top G_n(W)u
&=u^\top\left[
\frac1n X_S(W)^\top X_S(W)
-\tau_nB_S^\top B_S
\right]u=\frac{1}{n}(B_Su)^\top W^\top A^\top AW(B_Su)-\tau_n(B_Su)^\top(B_Su)\nonumber\\
&=\frac{1}{n}p^\top W^\top A^\top AWq
-\tau_n p^\top q.
\end{align}
where $p=q=B_Su$. Lemma~\ref{lem:bilinear} is usable for $u^\top G_n(W)u$.

 We consider a deterministic $\frac{1}{4}$-net
$\mathcal{N}\subset\mathbb{S}^{k-1}$, with cardinality at most $9^k$. We claim $\|M\|_\mathrm{op}\leq2\max_{u\in\mathcal{N}} |u^\top Mu|$, for a symmetric $M\in\mathbb{R}^{k\times k}$. Let $A_\mathcal{N} = \max_{u\in\mathcal{N}}|u^\top Mu|$, and let $x\in\mathbb{S}^{k-1}$, with $u\in \mathcal{N}$. Then,
\begin{equation}
\begin{aligned}
\left|x^\top M x-u^\top Mu\right|&\leq \left|(x-u)^\top Mx\right|+\left|u^\top M(x-u)\right|
\end{aligned}
\end{equation}
Considering the Cauchy-Schwarz inequality for $\left|(x-u)^\top Mx\right|$ and $\left|u^\top M(x-u)\right|$ we have
\begin{align}
\left|(x-u)^\top Mx\right|
&\leq 
 \|x-u\|_2\|M\|_\mathrm{op}\|x\|_2\leq \frac{1}{4}\|M\|_\mathrm{op},\nonumber\\
\left|u^\top M(x-u)\right|
&\leq \|x-u\|_2\|M\|_\mathrm{op}\|u\|_2\leq \frac{1}{4}\|M\|_\mathrm{op}.
\end{align}

Thus, $|x^\top Mx|\leq |u^\top Mu|+\frac{1}{2}\|M\|_\mathrm{op}\leq A_\mathcal{N}+\frac{1}{2}\|M\|_\mathrm{op}$. Considering the definition of the operator norm $\|M\|_\mathrm{op} = \sup_{\|x\|_2=1} |x^\top Mx|$ we deduce that $\|M\|_\mathrm{op}\leq 2A_\mathcal{N}$. Choosing $M = G_n(W)$ proves our claim. Let $t>0$, then,
\begin{align}\label{eq:beforefinalconct}
\mathbb{P}\left(\|G_n(W)\|_\mathrm{op}>t\right)
\leq \mathbb{P}\left(
\max_{u\in\mathcal{N}}|u^\top G_n(W)u|>\frac{t}{2}
\right)
\leq \sum_{u\in\mathcal{N}}
\mathbb{P}\left(|u^\top G_n(W)u|>\frac{t}{2}\right).
\end{align}
By Lemma~\ref{lem:bilinear}, and combining it with \eqref{eq:beforefinalconct}, we have
\begin{align}
\mathbb{P}\left(\|G_n(W)\|_\mathrm{op}>t\right)&\leq 9^k C_0e^{-c_0n\min\left(\frac{t^2}{4},\frac{t}{2}\right)}\nonumber\\
\mathbb{P}\left(\|G_n(W)\|_\mathrm{op}>t\right)&\leq C_0e^{k\log(9)-c_1n\min\left(t^2,t\right)},
\end{align}
where $c_1 = \frac{c_0}{4}$. Since $t>0$ is fixed, we have $\min(t^2,t)>0$, and by Assumption~\ref{ass:Sparse}, $k\log(9)-c_1n\min(t^2,t)\to -\infty,$ therefore, $\mathbb{P}\left(\|G_n(W)\|_\mathrm{op}>t\right)\to 0$ for all $t>0$, thus \eqref{eq:ineq1lemm6} holds. 
 To prove \eqref{eq:ineq2lemm6}, first recall that $g_n(W) = \frac1n X_S(W)^\top \xi\in\mathbb{R}^k$ from Section~\ref{subsec:Gngn} and let $u\in\mathbb{R}^k$, such that $\|u\|_2 =1$. Consider
\begin{align}
u^\top g_n(W)
=u^\top \frac1n X_S(W)^\top \xi
=\frac 1n (B_Su)^\top W^\top A^\top \xi
=\frac 1n p^\top W^\top A^\top \xi,
\end{align}
where $p \coloneqq B_Su$. Since $\|p\|_2\leq C_B$, Lemma~\ref{lem:bilinear} is usable. We use the same $\frac{1}{4}$-net as before and claim $\|g\|_2\leq \frac 43 \max_{u\in\mathcal{N}}|u^\top g|$ for all $g\in\mathbb{R}^k$.  Let $x = \frac{g}{\|g\|_2}\in\mathbb{S}^{k-1}$, and let $u\in\mathcal{N}$, consider
\begin{align}
\|g\|_2
&=x^\top g=u^\top g+(x-u)^\top g\leq |u^\top g|+|(x-u)^\top g|\leq |u^\top g|+\frac{1}{4}\|g\|_2\leq \frac 43 \max_{u\in\mathcal{N}}|u^\top g|,
\end{align}
where choosing  $g = g_n(W)$ proves our claim. Note that if $\|g_n(W)\|_2> t$, then $\max_{u\in \mathcal{N}} \left|u^\top g_n(W)\right|>\frac{3t}{4}$.

Let $\mathcal{E}_M = \left\{\|\xi\|_2\leq M\sqrt{n}\right\}$, be a high-probability noise event, defined for a fixed $M>0$. Appealing to  Lemma~\ref{lem:bilinear} results in
\begin{align}
\mathbb{P}\left(
\left|u^\top g_n(W)\right|>\frac 3 4 t
\mid \xi
\right)\leq 2\exp\left(
-\frac{c_0n(3t/4)^2}{\|\xi\|_2^2/n}
\right)
\leq 2\exp\left(-\frac{c_1nt^2}{M^2}\right),
\end{align}
for some constant $c_1>0$, and $\xi\in\mathcal{E}_M$. Taking the union bound over $\mathcal{N}$, for all $\xi\in \mathcal{E}_M$, gives
\begin{align}
\mathbb{P}\left(\|g_n(W)\|_2> t\mid \xi\right)&\leq \mathbb{P}\left(\max_{u\in\mathcal{N}}\left|u^\top g_n(W)\right|>\frac 3 4 t\middle| \xi\right)\leq \sum_{u\in\mathcal{N}}\mathbb{P}\left(\left|u^\top g_n(W)\right|>\frac 3 4 t\middle| \xi\right)\nonumber\\
&\leq 2\cdot9^k\mathrm{e}^{-\frac{c_1nt^2}{M^2}}.
\end{align}
We want to prove that $\limsup_{n\to \infty}{\mathbb{P}(\|g_n(W)\|_2>t}) = 0$. By basic probability rules, we have
\begin{equation}
{\mathbb{P}(\|g_n(W)\|_2>t})\leq \mathbb{P}\left(\|g_n(W)\|_2>t,\mathcal{E}_M\right)+ \mathbb{P}(\mathcal{E}_M^c).
\end{equation}
By $\mathbb{P}\left(\|g_n(W)\|_2>t, \mathcal{E}_M\right) = \mathbb{E}\left[\mathbf{1}_{\mathcal{E}_M}\mathbb{P}\left(\|g_n(W)\|_2>t\mid \xi\right)\right]$, we have
\begin{equation}
\mathbb{P}\left(\|g_n(W)\|_2>t, \mathcal{E}_M\right) \leq C\mathrm{e}^{k\log(9)-\frac{c_1nt^2}{M^2}}
\end{equation}
 Since $t>0$, and $M<\infty$ are fixed, and by Assumption~\ref{ass:Sparse}, $k\log(9)-\frac{c_1nt^2}{M^2}\to -\infty,$ therefore, 
\begin{equation}
    \mathbb{P}\left(\|g_n(W)\|_2>t, \|\xi\|_2\leq M\sqrt{n}\right)\xrightarrow{}0
\end{equation}
Using Markov's inequality, we have
\begin{align}
\mathbb{P}(\mathcal{E}_M^c)
=\mathbb{P}\left(\|\xi\|_2>M\sqrt{n}\right)
=\mathbb{P}\left(\|\xi\|_2^2>M^2n\right)
\leq \frac{\mathbb{E}\|\xi\|_2^2}{M^2n}
=\frac{\sigma^2}{M^2}.
\end{align}
Therefore, we observe $\limsup_{n\to \infty}{\mathbb{P}(\|g_n(W)\|_2>t}) = 0$ for arbitrarily large $M>0$. Hence, \eqref{eq:ineq2lemm6} holds, which completes the proof.
\end{proof}

\section{Proof of Lemma~\ref{lem:Lemma14}}\label{app:appG}

\begin{proof}

Throughout the proof, the objective function in \eqref{eq:detlasso} is defined following the same notation as Section~\ref{sec:nontrivialsolution}. We first show that $F(\beta)$ is $\kappa$-strongly convex, therefore guaranteeing uniqueness and stability of the solution. Let $q_n(\beta)  \coloneqq \frac{1}{2}
\left(\beta-\beta_S^\star\right)^\top
Q_n
\left(\beta-\beta_S^\star\right) $. For all $\beta, \beta'\in \mathbb{R}^k$, and $t\in[0,1]$, let $\beta_t \coloneqq (1-t)\beta+t\beta'$, then we have
\begin{align}\label{eq:subtract1}
q_n(\beta_t)
&=
\frac{1}{2}
\bigl((1-t)a+tb\bigr)^\top
Q_n
\bigl((1-t)a+tb\bigr)=
\frac{1}{2}(1-t)^2a^\top Q_na
+
t(1-t)a^\top Q_nb
+
\frac{1}{2}t^2b^\top Q_nb,
\end{align}
where $a  \coloneqq  \beta-\beta_S^\star$ and $b \coloneqq \beta'-\beta_S^\star$. On the other hand, we have
\begin{equation}\label{eq:subtract2}
(1-t)q_n(\beta)+tq_n(\beta')
=
\frac{1}{2}(1-t)a^\top Q_na
+
\frac{1}{2}tb^\top Q_nb.
\end{equation}
Subtracting \eqref{eq:subtract1} from \eqref{eq:subtract2} and some algebraic manipulations yields
\begin{equation}\label{eq:qnstrongconvex}
(1-t)q_n(\beta)+tq_n(\beta')-q_n(\beta_t)
=
\frac{t(1-t)}{2}
(\beta-\beta')^\top
Q_n
(\beta-\beta').
\end{equation}
By Assumption~\ref{ass:lassoUniqueness}, we deduce that $Q_n\succeq
 \kappa I_k$, which means that for all $u\in \mathbb{R}^k$ we have $u^\top Q_n u \geq \kappa\|u\|_2^2$. By rearranging \eqref{eq:qnstrongconvex} we deduce that $q_n(\beta)$ is a $\kappa$-strongly convex function. 
Note that the $\ell^1$ norm is also convex, thus
\begin{equation}
F(\beta_t)
\leq
(1-t)F(\beta)
+
tF(\beta')
-
\frac{\kappa}{2}
t(1-t)\|\beta-\beta'\|_2^2,
\end{equation}
which means $F(\beta)$ is $\kappa$-strongly convex, which results in uniqueness of $\hat{\beta}^\mathrm{det}$. Next, we prove the boundedness of $\|\hat{\beta}^\mathrm{det}\|_2^2$. We consider the optimality of $\hat{\beta}^\mathrm{det}$ which gives
\begin{equation}
F(\hat{\beta}^\mathrm{det})\leq F(0),
\end{equation}
which allows us to write
\begin{equation}\label{eq:boundboundbound}
\lambda\|\hat{\beta}^\mathrm{det}\|_1\leq F(\hat{\beta}^\mathrm{det})\leq \frac {\tau_n} 2\|B_S\beta_S^\star\|_2^2.
\end{equation}
To analyze the right-hand side, we recall that by the Assumptions $\|B_S\beta_S^\star\|_2\leq C_B\alpha$. Furthermore, $\tau_n = \frac 1 n \|A\|_F^2\leq \|A\|_\mathrm
{op}^2\leq C_A^2$. Combining this with \eqref{eq:boundboundbound} results in
\begin{equation}\label{eq:eq79}
\|\hat{\beta}^\mathrm{det}\|_1\leq \frac{C_A^2C_B^2\alpha^2}{2\lambda}.
\end{equation}
Now, because $\|\hat{\beta}^\mathrm{det}\|_2\leq \|\hat{\beta}^\mathrm{det}\|_1$, we deduce that $\sup_n \|\hat{\beta}^\mathrm{det}\|_2<\infty$. 

To prove \eqref{eq:stufftoproveforuniqueness}, let $\hat{Q}_n(W) \coloneqq \frac 1 n X_S(W)^\top X_S(W)$, and $\Delta_n(W) \coloneqq  \hat{Q}_n(W)-Q_n$. We must prove
\begin{equation}
\mathbb{P}\left(\lambda_\mathrm{min}\left(\hat{Q}_n(W)\right)\geq \frac \kappa 2\right)\to 1.
\end{equation}
By appealing to Weyl's inequality and Assumption~\ref{ass:lassoUniqueness} we have
\begin{equation}
\lambda_\mathrm{min}\left(\hat{Q}_n(W)\right)\geq \kappa - \|\Delta_n\|_\mathrm{op}.
\end{equation}
Now, we define an event $\mathcal{E}_n \coloneqq \left\{\|\Delta_n(W)\|_\mathrm{op}\leq \frac \kappa 2\right\}$ as a helper that gives
\begin{align}
\mathbb{P}\left(\lambda_\mathrm{min}\left(\hat{Q}_n(W)\right)\geq \frac \kappa 2 \right)\geq \mathbb{P}
\left(\mathcal{E}_n\right)
=1-\mathbb{P}\left(\|\Delta_n(W)\|_\mathrm{op}>\frac \kappa 2\right).
\end{align}
From Lemma~\ref{lem:active-conc} we know $\|\Delta_n(W)\|_\mathrm{op}\xrightarrow{\mathbb{P}}0$ and so \eqref{eq:stufftoproveforuniqueness} holds.
\end{proof}

\section{Proof of Lemma~\ref{lem:inactive-conc}}\label{app:appH}

\begin{proof}

By Lemma~\ref{lem:bilinear} with $p=b_j$ and $q = B_Sh_n$, we have
\begin{align}
\mathbb{P}\left(
\max_{j\in S^c}
\left|\frac{1}{n}X_j(W)^\top X_S(W)h_n
-\tau_n b_j^\top B_Sh_n\right|>t
\right)&\leq
\sum_{j\in S^c}
\mathbb{P}\left(
\left|\frac{1}{n}X_j(W)^\top X_S(W)h_n
-\tau_n b_j^\top B_Sh_n\right|>t
\right)\nonumber\\
&\leq C \mathrm{e}^{\log(d)-cn\min(t^2,t)} \to 0,
\end{align}
which is \eqref{eq:inactive-signal-conc}.

Let  $a_j(W)  \coloneqq \frac1n X_j(W)^\top X_S(W)-\tau_nb_j^\top B_S$, and $u\in\mathbb{S}^{k-1}$, with $p=b_j$, and $q=B_Su$, then
\begin{equation}
a_j(W)u = \frac 1n p^\top W^\top A^\top AWq - \tau_n p^\top q.
\end{equation}
We apply Lemma~\ref{lem:bilinear}, for $t>0$ which results in
\begin{equation}
\mathbb{P}\left(|a_j(W)u|>t\right) \leq Ce^{-cn\min(t^2,t)}.
\end{equation}
We use the same $\frac 1 4$-net approach used in Lemma~\ref{lem:active-conc}, which allows us to write
\begin{align} \mathbb{P}\left(\max_{j\in S^c}\|a_j(W)\|_2>t\right)&\leq \mathbb{P}\left(\max_{j\in S^c}\max_{u\in \mathcal{N}}|a_j(W)u| >\frac 3 4 t\right)\leq\sum_{j\in S^c}\sum_{u\in \mathcal{N}}\mathbb{P}\left(|a_j(W)u|>\frac 3 4 t\right)\nonumber\\
&\leq Q\mathrm{e}^{k\log(9)+\log(d)-cn\min(\frac{9t^2}{16},\frac{3t}{4})},
\end{align}
for some constant $Q$. Therefore, we deduce that 
\begin{equation}
\max_{j\in S^c}\left\|\frac{1}{n}X_j(W)^\top X_S(W) - \tau_n b_j^\top B_S\right\|_2 \xrightarrow{\mathbb{P}}0.
\end{equation}
Now, by the bound $\max_{j\in S^c}\|\tau_n b_j^\top B_S\|_2 \leq C_A^2 C_B^2$, we obtain \eqref{eq:inactive-cross-op}.

In the final step, to prove \eqref{eq:inactive-noise-conc}, we consider $\mathcal{E}_M$ as in Lemma~\ref{lem:active-conc} and deduce that
\begin{align}\label{eq:neededbound}
\mathbb{P}\left(
\max_{j\in S^c}
\left|\frac{1}{n}X_j(W)^\top\xi\right|>t,
\mathcal{E}_M
\right)
= \mathbb{E}\left[\mathbf{1}_{\mathcal{E}_M}
\mathbb{P}\left(
\max_{j\in S^c}
\left|\frac{1}{n}X_j(W)^\top\xi\right|>t \Big{|} \xi
\right)\right]
\leq  2d \mathrm{e}^{-\frac{cnt^2}{M^2}}\to 0,
\end{align}
for every fixed $M<\infty$. Furthermore, 
\begin{align}
\mathbb{P}\left(
\max_{j\in S^c}
\left|\frac{1}{n}X_j(W)^\top\xi\right|>t
\right)
\leq
\mathbb{P}\left(
\max_{j\in S^c}
\left|\frac{1}{n}X_j(W)^\top\xi\right|>t,
\mathcal{E}_M
\right)
+\mathbb{P}(\mathcal{E}_M^c).
\end{align}
Thanks to Markov's inequality and \eqref{eq:neededbound} we obtain \eqref{eq:inactive-noise-conc} and conclude the proof.
\end{proof}

\end{document}